\documentclass{article}[12pt]
\usepackage{amsmath}
\usepackage{amsfonts,bm}

\usepackage{color}

   \newtheorem{theorem}{Theorem}[section]
   \newtheorem{lemma}[theorem]{Lemma}
   \newtheorem{proposition}[theorem]{Proposition}
   
   \newtheorem{condition}[theorem]{Condition}
   \newtheorem{defn}[theorem]{Definition}

   \newenvironment{proof}[1][Proof]{\begin{trivlist}
   \item[\hskip \labelsep {\bfseries #1}]}{\end{trivlist}}
   \newenvironment{definition}[1][Definition]{\begin{trivlist}
   \item[\hskip \labelsep {\bfseries #1}]}{\end{trivlist}}
   
   \newenvironment{remark}[1][Remark]{\begin{trivlist}
   \item[\hskip \labelsep {\bfseries #1}]}{\end{trivlist}}

   \newcommand{\qed}{\nobreak \ifvmode \relax \else
      \ifdim\lastskip<1.5em \hskip-\lastskip
      \hskip1.5em plus0em minus0.5em \fi \nobreak
      \vrule height0.75em width0.5em depth0.25em\fi}

\advance \topmargin by -\headheight
\advance \topmargin by -\headsep
     
\evensidemargin \oddsidemargin
   \numberwithin{equation}{section}

\newcommand{\pe}{\prime}

\newcommand{\cA}{\mathcal{A}}
\newcommand{\cB}{\mathcal{B}}
\newcommand{\cC}{\mathcal{C}}
\newcommand{\cD}{\mathcal{D}}
\newcommand{\cE}{\mathcal{E}}
\newcommand{\cG}{\mathcal{G}}

\newcommand{\cI}{\mathcal{I}}

\newcommand{\cR}{\mathcal{R}}    
\newcommand{\cS}{\mathcal{S}} 
\newcommand{\cT}{\mathcal{T}}

\newcommand{\cV}{\mathcal{V}}
\newcommand{\cW}{\mathcal{W}}

\newcommand{\bnu}{{\bm{\nu}}}
\newcommand{\bth}{{\bm{\theta}}}

\newcommand{\bA}{{\bf A}}
\newcommand{\bB}{{\bf B}}

\newcommand{\bL}{{\bf L}}
\newcommand{\bM}{{\bf M}}

\newcommand{\bT}{{\bf T}}
\newcommand{\bU}{{\bf U}}
\newcommand{\bW}{{\bf W}}

\newcommand{\bbB}{\mathbb{B}}

\newcommand{\bbD}{\mathbb{D}}

\newcommand{\bbN}{\mathbb{N}}

\newcommand{\bbP}{\mathbb{P}}

\newcommand{\bbR}{\mathbb{R}}    
 
\newcommand{\bbT}{\mathbb{T}}

\newcommand{\bbZ}{\mathbb{Z}}

 \newcommand{\ul}{\underline}
 \newcommand{\ol}{\overline}
 \newcommand{\wt}{\widetilde}
 \newcommand{\wh}{\widehat}

 \newcommand{\tn}{\textnormal}
 
\newcommand{\be}{\begin{equation}} 
\newcommand{\ee}{\end{equation}}
\newcommand{\beq}{\begin{eqnarray}}
\newcommand{\eeq}{\end{eqnarray}}

\DeclareMathOperator{\dist}{dist}

\newcommand{\ei}{{\bf e}_i}   

\newcommand{\lra}{\leftrightarrow}   

\newcommand{\cc}{\subset\subset}

\begin{document}

\title{Stretched exponential decay of correlations in the quasiperiodic continuum percolation model}
\author{Rajinder Mavi}
\maketitle
{\let\thefootnote\relax\footnote{This work was supported by the Institute of Mathematical Physics at Michigan State University and NSF Grant DMS-1101578} }

\section{Introduction} 
The setting of the continuum percolation model is the product of a discrete `space-like' graph $\cG = (\cV,\cE)$ with a continuous `time-like' dimension $\bbR$. The model is similar to the famous contact process, with `deaths' arriving independently on the vertex lines $\cV \times \bbR$ and `bonds' arriving independently on the edge lines $\cE\times \bbR $. Unlike the contact process, the continuous dimension in this model is non-oriented, so if a point $(x,t)$ percolates to a point $(y,s)$ then the converse is true as well. There are a variety of hybrid discrete-continuous population models, for a review, see \cite{Grimmett2008}.

   As with traditional bond/site percolation models, the continuum percolation model with uniform death rate $\delta$ and bond rate $\lambda$ has a continuous phase transition from a low density phase $\lambda/\delta < \rho_c$ to a high density phase $\lambda / \delta > \rho_c$.  The low density phase is characterized by exponential decay of correlations \cite{BG1991}, 
   \[ P[(x,t)\tn{ and }(y,s)\tn{ belong to the same cluster}] < 
      e^{- \mu ( |x- y| + |t -s| )} \]
   and the high density phase is characterized by the almost sure existence of an unbounded cluster. At $\lambda / \delta = \rho_c$ the probability of an infinite cluster is zero \cite{BG1990}.
   
   It is interesting to ask whether this phase transition is preserved for non-uniform environments $\delta$ and $\lambda$. More formally, let $\delta: \cV \to (0,\infty)$, and $\lambda:\cE \to (0,\infty) $, and now define the death  rate $\kappa^{-1}\delta$ and bond rate $\kappa\lambda$. We ask whether the phase transition persists, and what is the effect on the phases,  is the effect on the phases, and existence of a phase transition  are the high and low density phases preserved in the parameter $\kappa$.

    It is natural to ask how stable these phases are with respect to non-uniform perturbations.  In particular we are interested in perturbations such that $ \inf\{ \delta_x \mid x\in \cV\} \cup \{\lambda^{-1}_{u} \mid u \in \cE \} = 0$. Under sufficiently mild perturbations, (where the value of $\delta$ is changed at a set of points with zero density in $\cV$) the phase transition value $\rho_c$ does not change, but the phase transition becomes discontinuous \cite{madras}.
 
  Moreover, it has been shown that {\it a} phase transition is preserved for randomly disordered parameters $\lambda $ and $\delta$. Klein \cite{klein94} showed, under mild conditions on the moments of parameters that a low density phase with exponential decay of correlations in the spatial direction exists. More formally, in that paper it was shown that, if $\delta$ and $\lambda$ are i.i.d. random values such that 
   \be \label{iidkleincondition} 
     E [\log (1 + \delta^{-1} )]^\beta + E [\log (1 + \lambda )]^\beta < \infty   \ee
    and 
    \be \label{kleinratio} E[ \log (1 + \lambda/\delta ) ]^\beta   \ee 
   is sufficiently small, where 
   \[\beta > 2d^2 ( 1 + \sqrt{ 1 + 1/d} + 1/2d ), \]
   then correlations decay as
   \[  P[(x,t)\tn{ and }(y,s)\tn{ belong to the same cluster}] < 
      C_{x,\delta, \lambda} e^{- \mu ( |x- y| + [\log (1 + |t-s|) ]^{\tau})}    \]
   for some $\tau > 1$. Note that if $\delta^{-1}$ and $\lambda $ are bounded below, and if (\ref{kleinratio}) is large then there is an infinite component with probability one, in that sense the phase transition is preserved for such models. As an alternative to random disorder, it is interesting to consider quasiperiodic disorder. The formulation of quasiperiodic disorder is stated below in Section \ref{main11}. Similar bounds on the bounds on the percolation probability hold for quasiperiodic disorder  \cite{jito}, with similar moment conditions on the sampling functions. We discuss this similarity in detail in \cite{mavi1}. On the other hand, for sufficiently strong disorder there is no phase transition. In the random case, sufficient concentration of $\delta$ near zero and heavy tails of $\lambda$ imply there is no phase transition depending on the moment behavior of $\lambda / \delta$ \cite{AKN}. In the quasiperiodic case, the sampling function defining $\delta$ vanishing sufficiently fast near a zero point implies there is no phase transition in the choice of $\lambda$.
   
   In this paper we will show the existence of an intermediate behavior in the low density phase due to the behavior of the sampling function near zero values. We show that for Diophantine disorder and power law  behavior near zero points, the correlation probabilities decay as
   \[ P[(x,t)\tn{ and }(y,s)\tn{ belong to the same cluster}] < 
      C_{x,\delta, \lambda} e^{- \mu ( |x- y| +  |t-s |^\tau )}      \]
   for some $\tau  > 0$.

   \subsection{Percolation Model}\label{sec:modelintro}
    Let consider the graph determined by the integer lattice $ \bbZ^d$ connected by nearest neighbors. The vertex set of the graph is 
    \be \label{vertex-int}  \cV := \bbZ^d. \ee
    The edge set is composed of the pairs of nearest neighbors $\{x,y\}$ so that $\|x - y\|_1 = 1$. The edges will be labeled by their midpoints: 
    \be\label{edge-int}  \cE := \left\{ \left. \frac{x+y}{2} \right| \tn{ for all } x,y\in \bbZ^d \tn{ so that }\|x-y\|_1 =1  \right\}.   \ee
    We state our proofs and carry out the analysis in the context of the integer lattice, but similar proofs can be carried out in any graph with bounded coordination number. The graph is defined as $\cG = \cV \cup \cE $. The model is defined on $\bL_\cG = \cG \times \bbR $.
   
   To each $x \in \bbZ^d $ we associate a rate parameter $ \frac{1}{\kappa}\delta_x > 0$ for a Poisson process of deaths.  To each edge $u \in \cE $ we associate a rate parameter $\kappa \lambda_u > 0$ for for a Poisson process of bonds. Here $\kappa  > 0 $ is a global tuning parameter. These processes take place on $ \bL_\cG$ which may be visualized as being embedded in $\bbR^{d+1}$.

       For each $x \in \cV$, for  each arrival $t$ of the death process on $ x\times \bbR$ we delete the point $x\times t$ from $\bbZ^d\times \bbR$. For each $u \in \cE $, for each arrival $t$ of the bond process on $u \times \bbR $  we identify points $(x,t)$ and $(y,t)$, where $u = (x+y)/2$ for nearest neighbors $x$ and $y$. The resulting percolation measure, which we denote by $P = P_{\kappa,\delta,\lambda}$, is the product measure over the  collection of Poisson measures.  Denote the set of realizations of the Poisson processes as $\Omega$. With probability one, each realization $\omega \in \Omega$ breaks $\cV\times \bbR $ into finitely many connected components. For $x,y\in \cV$, $t,s \in \bbR$, denote the event that $( x,t)$ belongs to the same component as $(y,s)$ by the notation $\{ (x,t) \lra (y,s) \}$. Note in particular, that the real dimension in this model is not directed as it is in the well known contact process, thus if a point $(x,t)$ `percolates to' a point $(y,s)$ then conversely, $(y,s)$ percolates to $(x,t)$.

      The model in the uniform case, $\delta_x \equiv \delta $ ( for all $x  \in \cV$) and $\lambda_{u} \equiv \lambda$ (for all $ u \in \cE$),  (and $\kappa = 1$) has been well studied and is known to exhibit long and short range phases. For all $d \geq 1$, if $\frac\lambda\delta \leq \rho_d$ the system is in the short range phase with exponentially decaying two point percolation $P[ \{ (x,t)\lra (y,s) \}] < e^{- \lambda \mu_\rho (\|x - y\|_1 + |t - s|)}$. For $\frac\lambda\delta >  \rho_d$ the system is in a long range phase with positive probability of an infinite cluster at the origin. 
      
     It is not hard to see that the scaling $ (\delta , \lambda ) \to (s\delta , s\lambda)$  does not effect the phase of the model, its only effect is to dilate the real dimension by a factor of $s$. This remains true in the disordered case. Thus, the environment $(\kappa^{-1} \delta, \kappa\lambda )$ is equivalent to $ (\delta , \kappa^2 \lambda)$ under rescaling by a factor of $\kappa$. The latter formulation was used in \cite{jito} for the study of quasiperiodic death rate with constant, and small, bond rate.

 Finally, let us observe that, for disordered parameters $\delta,\lambda$ we may define  local upper and lower densities of the environment. The upper (and respectively lower) density at $x$ are defined as
      \begin{align}
         \ol{\rho_x} =\kappa^2\frac{ \max_{y: \|x - y\|= 1}  \lambda_{x,y}   }{\delta_x};
         \hspace{.5in}
         \ul{\rho_x}   =  \kappa^2\frac{ \min_{y: \|x - y\|= 1}  \lambda_{x,y} }{\delta_x}.
      \end{align}
       If $\ol{\rho_x } \leq \rho_d$ (respectively, $\ul{\rho_x } > \rho_d$) uniformly for $x\in \bbZ^d$ then the system is in the short (respectively long) range phase. However, such global results do not apply when if there exists $A \subset \bbZ^d$ so that $A$ and $A^c$ are infinite and   $ \ol{\rho_{x} } \leq \rho_d <   \ul{\rho_{y}}$ for every $x\in A$ and $y\in A^c$. It is precisely these regimes we are interested in as we describe in the following section.

   \subsection{Main definitions and results}\label{main11}

We will begin by constructing quasiperiodic fields $\delta$ and $\lambda$ on $\cV$ and $\cE$ respectively. The fields are constructed by introducing a quasiperiodic dynamical system and then define the field by sampling along the orbit of the dynamical system.

Concretely, we will consider quasiperiodic shifts of $d+1$ multidimensional tori. We will write $\bbT = \bbR/\bbZ = [0,1]/\{0,1\}$ for the one dimensional torus, and, for any $\nu$, $\bbT^\nu = \bbT \times \cdots \times \bbT$ for the $\nu$-dimensional torus (from here on the $\nu$-torus). Shifts of the $\nu$-torus will be defined by {\it Diophantine} matrices which we will now define.
\begin{defn}[$\zeta$-Diophantine]
 A matrix $\bM \in \bbR^{\nu \times d}$ is $\zeta$-Diophantine if there is some $C_\zeta > 0$ so that for all $x \in \bbZ^d\setminus \{0\}$ and $\theta \in \bbT^\nu$
 \be\label{dio}
    d(\bM x + \theta,\theta) \geq \frac{C_{\zeta}}{|x|^{\zeta}}
    \ee
    where $d(\cdot,\cdot)$ denotes the distance on $\bbT^\nu$.
    Finally, we simply say $\bM$ is Diophantine if it is is $\zeta$-Diophantine for some $\zeta > 0$.
\end{defn}

  For $i = 0,1,...,d$ let $\nu_i\in \bbN$ be the dimension of the $i^{th }$ torus: $\bbT^{\nu_i}$. For each $i = 0,1,..,d$, let $\bM_i \in \bbR^{\nu_i\times d}$ be a Diophantine matrix. Thus we have $d+1$ quasiperiodic processes, defined, for $i = 0,1,..,d$, as 
  \[\bT_i^x \theta = \bM_ix + \theta_i \]
for $x \in \bbZ^d$ and phase $\theta_i \in \bbT^{\nu_i}$.

Now we introduce sampling functions which define the fields. Let  $\cC(\bbT^{\nu})$ be the real continous functions on $\bbT^\nu$. The sampling  functions will belong to spaces of the form
\[ \cC_{fin}^{+}(\bbT^{\nu}) = \{ h\in \cC(\bbT^{\nu}) \mid  h\geq 0 ; 0 \leq  |h^{-1}(0)| < \infty \}.   \] 
We will restrict our attention to sampling functions which have a power law behavior at their zeros.

\begin{defn}[$\sigma$-admissable] A function $h \in  \cC_{fin}^{+}(\bbT^{\nu})$ is $\sigma$-admissable if for any $\theta \in h^{-1}(0)$, 
 \[  \limsup_{\epsilon \to 0} \sup_{\theta' : d(\theta',\theta ) < \epsilon} \frac{|\log h(\theta')|}{|\log  d(\theta',\theta)| } < \sigma \] 
\end{defn}

Finally we define the fields. Let $h_0 \in \cC_{fin}^{+}(\bbT^{\nu_0})$, given phase $\theta_0 \in \bbT^{\nu_0}$ define the death rate as
\[  \delta_{x} =  h_0(\theta_0 + \bM_0 x  ).  \]
For $i = 1,..,d$ let  $\ei \in \bbZ^d$ be the vector with a 1 in the $i^{th }$ position and 0s at all other positions. For $i = 1,..,d$, the sampling function $h_i \in \cC_{fin}^{+}(\bbT^{\nu_i})$, and initial condition $\theta_i \in \bbT^{\nu_i}$ define bonds at edges $\cV + \frac12\ei $. For each $x \in \cV$ define the bond rate at the edge $ x + \frac12 \ei$ by
   \[   \lambda_{x + \frac12\ei}  
       =  \frac1{  h_i(\theta_i + \bM_i x  ) } \]

Let us summarize the above construction by refering to the family $\{(\bbT^{\nu_i},\bT_i,h_i)\}_{i = 0}^d$ as an environment process and $\{\theta_i\}_{i=0}^d$ as the initial phases. We summarize conditions on the environment process in the following condition.

    \begin{condition}[$(\zeta,\sigma)$-proper]\label{properenv}
      We say an enviroment process $\{(\bbT^{\nu_i},\bT_i,h_i)\}_{i = 0}^d$ is $(\nu,\zeta,\sigma)$-regular if, for $i = 0,1,...,d$:
      \begin{itemize} 
         \item[1.)] $\nu \leq \min\{\nu_0,\nu_1,..,\nu_d \}$.
         \item[2.)] For each $i$, the matrix $\bM_i \in \bbR^{\nu_i\times d}$ defining $\bT_i$ is $\zeta$-Diophantine. 
         \item[3.)] The sampling function $h_i\in \cC^+_{fin}(\bbT^{\nu_i})$ is $\sigma$-admissable.
      \end{itemize} 
    \end{condition}
   Given the sampling functions $\{h_i\}_{i=0}^d$, where $ h_i \in \cC_{fin}^+(\bbT^{\nu_i})$, let $R_i = |h^{-1}(0)|$  count the number of zeros. Define further counting parameters
    \be\label{Rvals}
     R_v := R_0 \tn{ and } R_e := R_1 + \cdots + R_d \tn{ and the sum } R = 2 \vee( R_v + R_e). \ee
    The parameters bound the resonances on each scale in the multiscale analysis which we discuss the parameters which we discuss in Section \ref{disc}.

   \begin{theorem}\label{th:diophantine} 
   	Suppose an environment process $\{(\bbT^{\nu_i},\bT_i,h_i)\}_{i = 0}^d$ is $(\nu,\zeta,\sigma)$-regular and $\mu > 0$ is fixed. Then for sufficiently small  $\kappa > 0$, there is some $C > 0$ so that for any 
   	  \be\label{tau0}
   	  \tau < 
   	    \frac{ C}{ 1+ R^2   \sigma\max\{\frac{d}{\nu} ,  \zeta \}}  \ee
   	   we have, for any $(x,t) \in \bbZ^d \times \bbR$ that there is a finite $C_x$ so that for all $(y,s) \in \bbZ^d \times \bbR$
   	\[   P[(x,t) \lra (y,s)] < C_x \exp\left( - \mu \max \{\|y - x\|, |s-t|^\tau\}  \right)   \]
   \end{theorem}
   The choice of $\tau $ in (\ref{tau0}) is not optimized, however, the constant $C$ can be chosen to be $C = \frac1{50}$. 
   
   \subsection{Discussion of the methods.}\label{disc}
   
   The proof of Theorem \ref{th:diophantine} proceeds by a multiscale analysis. We will summarize the argument in this section, as well as the organization of the paper.
   
   The multiscale analysis is an iterative argument proceeding over a sequence of length scales. The sequence is governed by a family of parameters which we will introduce now.
      Let 
   \[  K = \max\left\{\frac{d}{\nu} ,  \zeta \right\}  \] 
   and let $\alpha   $ be a parameter so that 
   \be \label{alpha00}
    \alpha > \sigma R K.
   \ee 
   Let  $\gamma $  be a parameter so that
   \be\label{gamma00}
      \frac{\alpha}{\sigma K} > \gamma > R
   \ee
   Now let $\tau$ satisfy
     \be \label{tau13}
      \frac{1 + R^2 K \sigma}{C} > 
    \frac1\tau  > \frac{\gamma}{\gamma - 1}
   \left( 1  +  2 \alpha R \frac{\gamma + R}{\gamma - R}\right)  \ee
   for $C = \tfrac1{50}$.
   To see that such a choice of parameters is possible, consider setting $\gamma = 2R$ and $ \alpha = 4 R K \sigma $ then
   \[
   \frac{1 + R^2 K \sigma}{C}
    > \frac1\tau > \frac{\gamma}{ \gamma -1}
   \left( 1  +  2 \alpha R \frac{\gamma + R}{\gamma - R}\right)
   \]
   Finally let $\eta $ be a parameter satisfying
   \be \label{eta00}
    \frac1\tau > \eta > \frac{\gamma}{\gamma - 1}
   \left( 1  +  2 \alpha R \frac{\gamma + R}{\gamma - R}\right).  \ee

   The initial spatial length scale $L_0$ must be chosen sufficiently large.
   We then define an increasing sequence of scales tuned by $\gamma  $,
   \be\label{Lscales}
   L_{k+1} = L_{k}^{\gamma} \tn{ so that } L_k = L_0^{\gamma^{k}}. 
   \ee
  Corresponding to each length scale, the sequence of  scales for the continuous dimension are defined by
  \be \label{Tscales} T_k = L_k^\eta  \ee
   so that decay in the spatial dimension at scale $L_k$ is matched by decay in the continuous dimension at scale $T_k$.
   The parameter $\kappa$ is fixed at the initial scale $L_0$, let 
   \be\label{kappa}  \kappa = C_\kappa \epsilon_0  \ee
   where $ C_\kappa  = C_\kappa(d,\mu_0) $ is a sufficiently small constant depending on the dimension and the desired decay rate $\mu_0$. With the choice of $C_\kappa$,  $\ol{\rho_x} << \rho_d $ over `most' sets of radius $L_0$. Over such sets the correlation probabilities decay at rate $\mu_0$, which we show Proposition \ref{prop:rhosmall}.

   As the length scale $L_{\cdot}$ increases, the low density condition on $\ol{\rho_x} $ is violated over sites within a radius $L_k$ of given $x_0$. We will denote these points as resonances. The objective of the multiscale analysis is to limit the effect  these resonances have on the exponential decay. 
   \begin{defn}[$\epsilon$-resonant]\label{epsres}
   We say a point $x \in \cV $  respectively an edge $u \in \cE $ is $\epsilon$-resonant if $\delta_x < \epsilon$, respectively $\lambda_{u} > \epsilon^{-1}$.
   \end{defn}   
   
   We associate a resonance scale $\epsilon_k$ with length scale $L_k$ for each $k = 0,1,2...$ tuned by a parameter $\alpha$:
   \be \label{eps-k}
     \epsilon_k = L_k^{- \alpha}.
   \ee
   We will see that resonances only affects the decay of correlations in the spatial directions to the extent of the decay rates $\mu_k$. Let $0 < \beta < 1 - \gamma^{-1}$, the decay rates  form a decreasing sequence
   \[   \mu_{k+1} = \mu_k ( 1 - L_{k+1}^{-\beta} )  \]
   with a limiting lower bound $\mu_\infty = \inf_k \mu_k$. On the other hand, exponential decay will not hold in the continuous dimension, which leads to the scaling introduced in (\ref{Tscales}).

    In proper environments the number of resonances in any radius  $ L$ is uniformly bounded. Let the block of radius  $L$ at $x \in \cG$ be defined as
    \be \label{block} 
    \Lambda_L^\#(x) = \{ y \in \# : \|y - x\|_1 \leq L  \}  \ee
    for $\# = \cV, \cE, \cG$. In Proposition \ref{res-sep} we show Condition \ref{properenv} for proper environments controls  the number of  $\epsilon = L^{-\alpha}$ resonances in $\Lambda_{L^\gamma}^\#(x)$.

   The multiscale analysis utilizes the $L_k$ scale decay in the bulk of the $L_{k+1}$ scale regions. 
   When inducting from scale $L_k $ to $L_{k+1}$,
     sites which are $\epsilon_k$ resonant will require special care, and the assumption that they are not $\epsilon_{k+1} $ resonant. Sites which are $\epsilon_{k+1}  $ resonant will have to wait for higher scales to be included in the multiscale analysis. We will prove the induction on scale $L_0$ to $L_1$ in Proposition \ref{P:ind0} and the induction in the general case in Proposition \ref{P:indk}.
     
   To formalize the regions of a given scale, we denote a box of scale $(L,T)$ at a site $ (x,t)$ by
   \be \label{box}
   B_{L,T}^\#(x,t) :=  \Lambda_{L}^{\#}(x) \times \{ t - T, t + T \} 
   \ee
   An embedding of a block into $\bL_\cG$ is denoted by
   \be\label{eblock} 
     \Upsilon_{L}^\#(x,t) := \Lambda_{L}^\#(x) \times t.
   \ee
   The boundary of a box $B_{L,T}$ is partitioned into vertical boundary, denoted by sets $\Upsilon_L^\cV(x,t\pm T) $ and the horizontal boundary which may be written as
   \[ B_{L,T}^\cV(x,t) \setminus  B_{L-1,T}^\cV(x,t). \]
   In inducting regularity from scale $L_k$ to the box  $B_{L_{k+1},T_{k+1}}$, we consider the vertical and horizontal boundaries separately.
   
    To control percolation to the horizontal boundary we partition paths by their visits to the resonant sites. The filtration into the partition is achieved by the BK inequality  (\ref{BK}). The horizontal percolation on the initial scale and the general scale is stated as Propositions \ref{prop:Hbdry0} and \ref{prop:Hbdryk} respectively.

    The more difficult portion of the proof is control of the percolation to the vertical boundaries. Percolation to the upper vertical boundary requires a path passing through $ \Upsilon_{L_{k+1}}^\cV (x,nT_k) $ for $ n = 1,..,T_{k+1} / T_k$.  On the first scale, $\Lambda_{L_1}^\cG(x)$ the resonant set $\cR_0$  is composed of at most $R_v $ resonant sites  and  $R_e$ resonant edges. Percolation from $\Upsilon_{L_1}(x,(n-1)T_1)$ to $\Upsilon_{L_1}(x,nT_1)$ is broken up into resonant and non resonant  percolation, the bound on percolation is stated in Proposition \ref{prop:layer0}. Outside the resonant set the communication probability is relatively small (Proposition \ref{prop:res0c}), and  we will show percolation from $ \cR_0 \times [(n-1)T_{0}, nT_0]$ to $\Upsilon_{L_{1}}(x,nT_0) $ is bounded by $1 -(c \epsilon_1)^{R} $ (Propositions \ref{prop:res0a} and  \ref{prop:res0b}). Percolation through the $T_1/T_0$ layers will obtain sufficient decay. We carry out a similar proof in the general case in Proposition \ref{prop:layerk}. Again the resonant set contains $\epsilon_k$ resonant sites and edges. But, as we will use regularity at higher scales, we need to include $\epsilon_i$ resonances which are sufficiently close to the $\epsilon_k $ resonances. Again we utilize Proposition \ref{res-sep} to bound the number of those resonances. We therefore have  percolation from $\Upsilon_{L_{k+1}}(x,(n-1)T_k)$ to $\Upsilon_{L_{k+1}}(x,nT_k)$ is bounded by $ 1 - (c\epsilon_{k+1})^p $  for large enough $p$. Similar to the first step, requiring percolation through $T_{k+1}/T_k$ layers shows sufficient decay toward the vertical boundary.

   Finally, we prove Theorem \ref{th:diophantine} in Section \ref{proper reg}. In Proposition \ref{pr:BC}, we show for all $x$, and sufficiently large $k$ (depending on $x$) $\Lambda^\#_{L_k}(x)$ is not $\epsilon_k$ resonant which implies such sites are $\mu_k$ regular. Lastly,  we present the proof of Theorem \ref{th:diophantine} where we `smooth out' the decay over fixed scales $L_k$ to arbitrary large distances.

   \section{Details of the percolation model}\label{percbg}
    The family of Poisson processes are indexed by the elements of the graph $\cG = \cV \cup \cE $. The death and bond processes take place respectively on the spaces 
         \[  \bL_\cV =   \cV\times \bbR \hspace{1in}  
           \bL_\cE =   \cE \times \bbR,   \]
    which are embedded in Euclidean space
    \[  \bL_\cG = \bL_\cV \cup \bL_\cG \subset \bbR^{d+1}.  \]

    Denote the   set of realizations of the Poisson processes by $\Omega$, which can be identified as the space of locally finite subsets of $\bL_\cG $:
    \be
      \Omega = \{ \omega \subset \bL_\cG \mid \forall B \cc \bbR^{d+1},\ | B\cap \omega| < \infty  \}.
    \ee
   For $\omega \in \Omega$ we will denote the deaths as 
   \[ \bbD_\omega  =  \omega \cap \bL_\cV \]
   and the bonds as 
   \[ \bbB_\omega = \omega \cap \bL_\cE.\]
   The spaces $\bL_\#$, for $\# = \cV,\cE,\cG$, inherit  the topology from the ambient space $\bbR^{d+1}$. Given $\omega$ we let bonds introduce an equivalence relation on this topology. We define a topology $\cT_\omega$ on $ \bL_\omega  :=  \bL_\cV \setminus \bbD_\omega$ by an equivalence relation identifying points $(x,t) $ and $(x + \ei,t)$ for each $(x + \frac12 \ei,t) \in \bbB_\omega$. For $X \in \bL_\omega$ and $W \subset \bL_\cV $, we will write $C_{W,\omega}(X)$ for the connected component of $(\bL_\omega \cap W,\cT_\omega ) $ which contains $X$. For $X,Y \in \bL_\omega$, we will write $ X \lra_{W,\omega} Y $ if $ Y \in C_{W,\omega} (X)$. For the sake of simplicity of notation we will write 
   \[ 
       \{ X \lra Y | W \}  \equiv  \{  X \lra_{W} Y \}. \]

      \subsection{Geometry in $\bbZ^d \times \bbR$}
    
      Let us begin with sets in the graph $\cG = \cE \cup \cV$. It is often helpful to restrict to a subgraph $\cW \subset \cG$. Let us define vertex and edge sets with respect to $\cW$ for a subset $A \subset \cG$. The vertex set is defined as
      \be \label{verset}
        \cV_{|\cW}(A) = \{ x \in \cW\cap \cV \mid \exists u \in A \cap \cW
                   \tn{ so that }\|u-x\|_1 \leq 1/2 \},   \ee
      note $u$ may be an edge or vertex.
      The edge set with respect to $\cW$ incident to a set  $A \subset \cG$ is denoted by
        \be \label{edgeset} \cE_{|\cW}(A) = \{ u\in \cW \cap \cE    \mid \exists x\in A\cap \cW  \tn{ so that } \|u-x\|_1 \leq 1 \},  \ee
        which includes the edges internal to the set, as well as those edges connecting the set to its complement.
     Now we can extend  Definition \ref{epsres} to sets.
     \begin{defn}[$\epsilon$-resonant]\label{setres}
       A set $A \subset \cG$ is $\epsilon$-resonant if there is a site  $x \in \cV(A)$  or an edge $u \in \cE(A) $ which is $\epsilon$-resonant.
     \end{defn}

      The lattice boundary with respect to $\cW$ is defined, for a subset $A \subset \cG $ as edges which connect $A$ to the complement. Let
      \[  \wh \partial_{|\cW} A 
       = \{ u \in \cW \cap \cE \mid 
             \Lambda_{1/2}^\cV(u) \cap A\cap \cW \neq \emptyset ;  \Lambda_{1/2}^\cV(u) \cap A^c\cap \cW  \neq \emptyset  \}\]
     For sets $A \subset \cG$ the inner boundary with respect to $\cW$ is defined as 
    \[   \wh \partial_{\cW}^- A = 
    \{ \Lambda_{1/2}^\cV(u) \cap A\cap \cW  \mid  u \in  \wh \partial_{|\cW} A   \}    \]
      and the outer boundary is defined as
      \[
       \wh \partial_{|\cW}^+ A = 
       \{ \Lambda_{1/2}^\cV(u) \cap A^c\cap \cW  \mid  u \in  \wh \partial_{|\cW} A   \}   
      \]
      which is equivalent to the inner boundary of the complement.
      The distance function for sets in the lattice is defined as usual: for sets  $A, B \subset \cG$ 
      \[  \dist(A,B) := \min_{x\in A;y\in B} \|x - y\|_\infty.   \] 
     We say the  sets $A$, $B \subset \cG $ are $L$-intersecting if 
      $\dist(A,B) \leq L.$

       Now we will define objects in $\bL_\cG$ with respect to a subset $\bW\subset \bL_\cG$. For a given subset $U \subset \bL_\cV = \bbZ^d \times \bbR $  we define the horizontal and vertical boundaries of $U$ with respect to $\bW$. 
         The vertical boundary is defined as
         \begin{align} 
            \partial^\cV_{|\bW} U =    \{ (x,t) \in \bL_\cG \mid
             \forall \epsilon > 0\ \exists\ s_1,s_2 
             \tn{ so that } |s_i-t| < \epsilon ;
              (x,s_1) \in \bW\cap U ; (x,s_2 ) \in \bW \cap U^c  \}.
         \end{align}
     The horizontal boundary for a subset $U \subset \bL_{\cV}$ is defined as 
   \begin{align} 
     \partial^{\cE}_{|\bW} U =   
        \{ ( u ,t) \in \bL_\cE  \mid 
               \Upsilon_{1/2}^\cV(u,t) \cap U \cap \bW \neq \emptyset ; \Upsilon_{1/2}^\cV(u,t) \cap U^c \cap \bW \neq \emptyset  \}
         \end{align}
         Given a  subset $U \subset \bL_\cV$  let us denote
          the inner adjacent points as
         \[ 
           \partial_{|\bW}^{\cE-}U 
                      = \{ \Upsilon_{1/2}^\cV(u,t) \cap U \cap \bW \mid (u,t) \in \partial_{|\bW}^\cE U \}   \]
           and outer adjacent points as
         \[ 
           \partial_{|\bW}^{\cE+}U 
                      = \{    \Upsilon_{1/2}^\cV(u,t) \cap U \cap \bW \mid (u,t) \in \partial_{|\bW}^\cE U  \}.   \]
                         Finally define the total boundary as  $ \partial_{|\bW} U = \partial_{|\bW}^{\cE} U \cup \partial_{|\bW}^{\cV} U $ and the total inner and outer boundary as  $ \partial^\pm_{|\bW} U = \partial_{|\bW}^{\cE\pm} U \cup \partial_{|\bW}^{\cV} U $. In all notation, when $ \bW$ is dropped we set $\bW = \bL_\cG$.

      \subsection{Topology of the percolation space $\Omega$}\label{sec:topology}  
      As discussed in \cite{BG1991}, the proper topology for $\Omega$ is the Skorokhod topology, which is simply constructed by   topologizing  counting functions on $\cG \times [-t,t]$ for all $t > 0$. Let $\cS_t$ be the set of cadlag functions on $[-t,t]$, and let $\cS_{\cG,t}$ be cadlag functions on $\cG\times [-t,t]$. Construct functions $N_t$ counting the arrivals of $\omega$ in $\cG \times [-t,t]$. For each $\omega$ we construct a  counting  function $N_{t,\omega}\in \cS_{\cG,t}$ on $\cG \times [-t,t]$ so that at any point $(u,s) \in   \cG \times [-t,t]$ we have 
      \[\lim_{\epsilon \to 0} N_{t,\omega}(u, s-\epsilon) = \begin{cases} N_{t,\omega}(u,s) - 1 & \tn{ if } (u,s) \in \omega  \\ N_{t,\omega}(u,s) & \tn{ if } (u,s) \notin \omega  \end{cases}. \]
       We fix the counting process by setting  $N(u, -t) = 0$ for all $u \in \cG$, unless $ (u,-t) \in \omega$ in which case we set $N_{t,\omega}(u, -t) = 1$. Let $\cD_t$ be the class of strictly increasing continuous functions on $[-t,t]$ onto itself, and for $r \in \cD$ define
   \[  \|r\| = \sup_{s_1 \neq s_2}
        \left| \frac{r(s_1) - r(s_2)}{s_1 - s_2} \right|.\]
        For cadlag functions $x,y$ on $[-t,t]$ to $\bbR$ define
        \[ d_t(x,y) = \inf_{r\in \cD_t}
      \left\{\|r\| + \sup_{s \in [-t,t]} |x(s) - y(r(s))| \right\} \]
      which generates the Skorokhod topology for functions on $\cS_t$.  Now we define  the topology on  $\cS_{\cG,t}$, which is defined, for $x,y \in\cS_{\cG,t} $ as 
           \[ d_{\cG,t}(x,y) =  \sum_{u \in \cG} e^{-\|u\|_1} 
                 \frac{ d_t\left(x_u(\cdot), y_u(\cdot)\right) }
                        { 1 + d_t\left(x_u(\cdot), y_u(\cdot)\right) }.  \]
           Finally,  a distance function on $\Omega$ may be defined by
          \[ d_{\cG}(\omega,\omega') 
               =    \int_0^\infty  d_{\cG,t}(N_{t,\omega} ,N_{t,\omega'}) e^{-t}dt.
           \] 
            Now $d_\cG$ is a complete metric and generates a topology $\cT$ on $\Omega$, we write $\cB(\Omega) $ for the Borel $\sigma$-algebra generated by $\cT$.

       We are primarily interested in percolation events in this paper, i.e. events of the type
         \[  E_W({A,B}) := 
          \{A \lra_W B \}
          =\cup_{X\in A;Y \in B} \{ X \lra_W Y  \}  \]
         for  sets $A,B \subset \bL_\cV$. 
          Thus let us discuss the boundary of such events.
          Recall the definition of Haudorff distance between subsets of a metric space is defined by
          \[
           d_H(X,Y) = \inf\{\epsilon > 0 : B_\epsilon (X) \supset Y \tn{ and } X \subset B_\epsilon (Y) \}
          \]
           According to the construction of Skorohod topology $\omega_k \to \omega$ in $\cT$ iff in any open bounded set $W \subset \bbR^{d+1}$ we have $  d_H(\omega_k \cap  W, \omega \cap W)  \to 0.$ It is not hard to see that the boundary of a percolation event $ E_W$ is contained in the event $Z $ that the arrival times of two crossings or cuts coincide. Indeed if $\omega \in E_W \setminus Z$. Thus the boundary of $E_W$ is contained in $Z$.

      \subsection{Increasing events}
         We introduce a partial ordering over configurations $\omega \in \Omega$. The partial ordering is relevant to describing probabilistic inequalities on certain subsets of $\Omega$.

        Given two configurations $\omega , \omega' \in \Omega$ we write 
        \[\omega \leq  \omega' \hspace{.2in} \tn{ iff } \hspace{.2in} \bbB_{\omega } \subset \bbB_{\omega'} \tn{ and } \bbD_\omega \supset \bbD_{\omega'}. \]
       We say an event $X \subset \Omega$ is increasing if $\omega \in X$ and $\omega' \geq \omega$ imply $ \omega' \in X$. If $X^c$ is increasing we say $X$ is decreasing. If $X$ is either increasing or decreasing we say $X$ is monotone.

         The following is the well known FKG inequality in this context.
         \begin{theorem}[FKG inequality]
         	If $X,Y\subset \Omega$ are both increasing (or both decreasing)  and $P(\partial X) = P(\partial Y ) = 0$ then
         	\begin{align} \label{FK}
         	  P(X \cap Y ) \geq P (X) P(Y) . 
         	\end{align} 
         \end{theorem}
 The FKG inequality is an important lower bound on intersection events. On the other hand, there is no upper bound on intersection events. But, there is an upper bound on the circle product of events.

          For events $X,Y \in \Omega$, we define the product set $X \circ Y$ as the set with the following property.  
         $\omega \in X \circ Y $  if there is a set $W\subset \bbR^{d+1}$ so that for every $\omega_1, \omega_2 \in \Omega $ so that $ W \cap \bbD_{\omega_1}  = W\cap \bbD_\omega $,   $ W \cap \bbB_{\omega_1}  = W\cap \bbB_\omega $ and    $ W^c \cap \bbD_{\omega_2}  = W^c\cap \bbD_\omega $,   $ W^c \cap \bbB_{\omega_2}  = W^c \cap \bbB_\omega $ we have $\omega_1 \in X$ and $\omega_2 \in Y$. In words, this is to say requirements for the  events $X$ and $Y$  are satisfied on disjoint sets.

          The following is the BK inequality.
          
          \begin{theorem}[BK inequality]
          	If both $X,Y\subset \Omega$ are both positive  (or both negative) and $P(\partial X ) = P(\partial Y ) = P(\partial (X\circ Y))$ then 
          	\begin{align} \label{BK}
          	     P(X \circ Y) \leq P(X) P(Y). 
          	\end{align}
          \end{theorem}

           \begin{remark}
           	 For percolation events $X$ and $Y$ the product event $X \circ Y$ is an event characterized by the existence of two non-intersecting percolation  paths. It is not hard to see that the boundary of this event is contained in the event $Z$ that the arrival times of two crossings or cuts coincide, i.e. $\partial (X \circ Y) \subset Z$. As discussed in the above $\partial X, \partial Y \subset Z$. As $P(Z ) = 0$ we may apply  both (\ref{FK}) and (\ref{BK}) to percolation events.
           \end{remark}

       \subsection{Filter inequality}

          Our main application of the BK  inequality is to filter communication events by the last visit of a percolation path to a set containing the initial site of the path.  
          
          Integration on the space $\bL_\cG$ is defined by the embedding of the Lebesgue measure on $\bbR$ onto the sets $w \times \bbR$ for $w \in \cG$. For each $w \in \cG$ define the map $\iota_w(s) = (w,s) \in \bL_\cG$. For any $f : \bL_\cG \to \bbR$, and $ \bW \subset \bL_\cG $ define
          \[  \int_{U \in \bW} f(U) dU = 
              \sum_{w \in \cG}
              \int_{\iota_w^{-1}(\bW \cap [w\times \bbR]) } 
               f(\iota_w(s) ) ds .   \]
        Now we define a quantification of the total communication to the boundary. Given a set $A \subset V \subset \bL_\cV$ and a set  $\bW \subset \bL_\cV $, let us define the sum percolation to the boundary within $V$ as 
          \begin{align}\label{eq:Q}
           Q(A,V|\bW) := &  
           \int_{ U \in \partial^\cE_{|\bW} V }
            \lambda_{U}  P[A \lra V \cap \Upsilon_{1/2}^\cV (U)|V\cap \bW]
            dU  
            + \sum_{ X \in  \partial^\cV_{|\bW} V }
           P[A \lra X| V \cap \bW].
           \end{align}
          The formal term defined in (\ref{eq:Q}) is used for the following bound for communication within a set $\bW \subset \bL_\cV$.
          \begin{lemma}\label{lem:liebfilter}
          	Let $A \subset V $ and $B \subset V^c $,
          	then 
          	\begin{align}\label{eq:liebfilter}
       P\left[ A \lra B |\bW\right] 
        	\leq   & 
                 Q(A,V | \bW) 
           \sup_{ Y \in \partial^+_{|\bW} V} 
              P[ Y  \lra B|  \bW ]
          	\end{align}
          \end{lemma}
          \begin{proof} 
             Any path from $A$ to $B$ must pass through $\partial_{|W} V$ finitely many times (as $\Omega$ is locally finite). Thus, there is some `first passage' of the path out of  $V$. We will proceed by filtering the paths by the first passage step which may be through the edge or the vertex boundary.  We will begin with the edge boundary, then consider the vertex boundary. 
          
           The event of the first path through a bond $U \in \partial^{\cE}_{|\bW} V $ at a given position may be written as 
             \[  F_{U} :=  
              \{ A \lra  V \cap \Upsilon^\cV_{1/2}(U)|V \cap \bW  \} \circ \{ U \in \bbB   \}  \circ \{  V^c \cap \Upsilon^\cV_{1/2}(U) \lra B   |  \bW \}.  \] 
             Similarly, a path crossing out of $V$ for the first time through a point in the vertical boundary is written as  
             \[ F_X =  \{ A \lra X  |V \cap \bW  \} \circ \{  X \lra B   | \bW  \} . \]  
              Thus we have
              \[ \{ A \lra B | \bW \} \subset 
               \left( 
              \bigcup_{ 
                U \in \partial_{|\bW}^{\cE} V }   
               F_{ U } \right) 
              \cup
               \left(  \bigcup_{X \in \partial_{|\bW}^\cV V } F_{X} \right).  \]
               Now apply (\ref{BK}) to all events $F_\cdot$. For the final steps through the edge boundary  we obtain  
               \begin{align} \label{edg1}
                P  &     \left( 
              \bigcup_{ 
                U \in \partial_{|\bW}^{\cE} V }   
               F_{ U } \right) 
               \leq 
               \int_{
                U \in \partial_{|\bW}^{\cE} V  } 
                P ( A \lra  V \cap \Upsilon^\cV_{1/2}(U)|V\cap \bW  )  P(  V^c \cap \Upsilon^\cV_{1/2}(U) \lra B   | \bW ) P( U \in \bbB   )  dU
               \end{align}
               The density term, by definition of the Poisson process is $ P( U \in \bbB   )  dU = \lambda_U dU $.
               For the final step through the vertex boundary we have
               \begin{align}\label{ver1}
               P\left(  \bigcup_{ X \in \partial_{|\bW}^\cV V } F_X \right)     \leq 
                \sum_{ X \in \partial_{|\bW}^{\cV} V}
                P[A \lra X| V \cap \bW]   
                P[ X \lra B| \bW ].
               \end{align}
               Taking the supremum of $P[Y \lra B|\bW]$ over the boundary terms in (\ref{edg1}) and (\ref{ver1}) completes the proof.
               \qed
          \end{proof}
           
      As a first application of Lemma \ref{lem:liebfilter} we show exponential decay in a low density environment. The proof also prepares us for the methods of the multiscale analysis. Recall, the parameter $\kappa$ is used to tune the crossing and  cut rates and we are using the notation $P=P_{\kappa,\delta,\lambda}$. 
      \begin{proposition} \label{prop:rhosmall} 
      Suppose $ \bW \subset \bL_\cG $ is $\epsilon$-non-resonant. Let  $X = (x,t)$ be a point and $U$ be a set such that $X \notin U$. For any $0 < \mu < \infty$, there is a $C_\kappa( \mu, d ) > 0$ so that for $\kappa = C_\kappa \epsilon  $ we have 
       \begin{align}\label{eq:rhosmall} 
      	 P [ X \lra  U| \bW ]
      	  <
      	   \exp\left(- \mu  \inf_{(y,s)\in U}  
    \min \left\{ \|x-y\| ,\left\lfloor  |t-s| \right\rfloor \right\} \right).
       \end{align}
      \end{proposition}

      \begin{proof}
         First we bound
         $  Q(X, B_{0,1}(X) | \bW ) $ defined in
         \ref{eq:Q}. 
         Note that for $Y  \in \partial^\cV B_{0,1}(X) $
         \[ \{ X \lra Y |\bW \} \subset
           \left( \{ x \times (t,t+1) \cap \bbD \neq \emptyset \}  \cap  \{ x \times (t-1,t) \cap \bbD \neq \emptyset \} \} \cap  \{  \partial^\cE B_{0,1}(x,t) \cap \bbB = \emptyset \}  \right)^c   \]
         A direct calculation of the probability of the right hand side shows that
      \[  P\left( \{ X \lra Y |\bW \}\right)
          =  
            1 -  \left(1 - e^{-\frac{1}{\kappa} \delta_x }\right)^2 
                \prod_{u \in \cE(x)} e^{- 2 \kappa \lambda_u}   \]
        On the otherhand, for $Y \in \partial^\cE B_{0,1}(X)$ we bound  $ P\left( \{ X \lra Y |\bW \}\right) \leq 1 $ thus, 
         the $Q$ term is  bounded by
         \[   Q[(x,t), B_{0,1}(x,t) | \bW ] 
           \leq 2\kappa \sum_{u \in \cE(x) } \lambda_u
            + 2 - 2 (1 - e^{-\frac{1}{\kappa} \delta_x })^2 
                \prod_{u \in \cE(x)} e^{-2 \kappa \lambda_u}   \]
         Now use the bound $\delta > \epsilon$ and $\lambda < \epsilon^{-1}$, Then, for $C_\kappa$, as in (\ref{kappa}), chosen sufficiently small with respect to $\mu$ we have 
         \[  Q[(x,t), B_{0,1}(x,t) | \bW ] 
           \leq 4d C_\kappa 
            + 1 - (1 - e^{-\frac{1}{ C_\kappa} })^2 e^{- C_\kappa 2d}
            < 
            e^{-\mu}. \]
       Thus, applying Lemma \ref{lem:liebfilter} 
                we have
        \[  P((x,t) \lra U | \bW) \leq e^{- \mu}
          \sup_{ (x_1,t_1) \in  \partial_{|W}^+ B_{0,1}(x,t) }    
       P( (x_1,t_1) \lra U | \bW  ) \]
        Now we iteratively apply this argument 
        \[ n  = \min_{(y,s)} \max\{ \lfloor \|x - y\|_\infty  \rfloor, \lfloor |t-s| \rfloor \}\]
        times which completes the argument.
      	\qed
      \end{proof}

\section{Main}

    We use the terminology {\it resonances} to describe points which are locally in the high density regime. This term was used in \cite{CK1991,CKP1991} in analogy to role bare energies  near a target energy play in multiscale proofs of Anderson localization.

   \subsection{Regularity induction}
   Let us first introduce the concept of regularity.

 \begin{definition}
 	We say a site $x\in \bbZ^d$ is $(\mu,L)$ regular if, for any $t\in \bbR$,
 	\[  Q[ (x ,t); B_{L,L^\eta}(x,t)  ] < e^{ - \mu L}  \] 
 	Moreover, for $\bW \subset \bL_\cV $, a point $(x,t) \in \cB$ is $(\mu,L|\bW)$-regular if 
 	\[  Q[ (x ,t); B_{L,L^\eta}(x,t) |\bW  ] < e^{ - \mu L}  \]
 \end{definition} 
 Note that, if we prove regularity at $(x,0)$, we have, by translation invariance, regularity at $(x,t)$ for all $t\in \bbR$.
 Regularity on an initial scale $L_0$ is easy to obtain using the tuning parameter $\kappa$. Regularity on higher scales will follow from induction, which we describe in Condition \ref{ick} and Proposition \ref{P:indk}. 
 \begin{proposition}\label{prop:bulk}
     For any $0 < \mu_0 < \infty$, for sufficiently small  $C_{\kappa}> 0$, for $\kappa = C_{\kappa } \epsilon_0 $ and any
    $x \in \cV $ such that $\Lambda_{L_0}(x)$ is $\epsilon_0$ non-resonant  then $x$ is $(\mu_0,L_0)$ regular.
      Moreover, if $\bW \subset \bL_{\cG}$ is $\epsilon_0$ nonresonant then for any $(x,t) \in \bW$, 
      \begin{align}\label{eq:bulk} 
      	 P [ (x,0)\lra  U| \bW]
      	  \leq \exp\left\{ - 2\mu_0 \inf_{(y,s)\in U}  \max \left\{ \left\lfloor  |t-s|  \right\rfloor, \|x-y\|  \right\}\right\}.
      \end{align}
 \end{proposition}	

   \begin{remark}
   	Note in the above theorem, the length of the initial scale $L_0$ is chosen first and then the parameter $\kappa$ is selected (as $\epsilon_0 = L_0^{-\alpha}$) with respect to $L_0$. Proposition \ref{prop:bulk} will follow from Proposition \ref{prop:rhosmall}.
   \end{remark}   

   \begin{proof}\
   	  Given $\mu_0$, the second statement follows by taking $C_\kappa $ sufficiently small to apply Proposition \ref{prop:rhosmall}, which  immediately obtains (\ref{eq:bulk}).  Now to obtain regularity, apply Proposition \ref{prop:rhosmall} to points on the boundary of $B_{L_0,T_0}(x)$. We have, for $(y,\pm T_0)\in \partial_{|W}^\cV B_{L_0,T_0}$ 
   	\[ P[(x,0)\lra(y, \pm T_0) ] 
   	   < e^{- 2\mu_0 T_0 }. \]
    For $(y,s) \in \partial^{\cE-}_{|W} B_{L_0,T_0}(x)$
   \[ P[(x,0)\lra(y, s) ] 
   	  < e^{- 2\mu_0L_0}. \]
   	    Applying the definition (\ref{eq:Q}) we have,
     \[   Q[(x,0) ; B_{L_0,T_0}(x)] \leq  
         \left(   C_dL_0^d  e^{ -  \mu_0( 2 T_0  -  L_0 )}   +  C_d L_0^{\eta + d-1} e^{ -\mu_0 L_0} \right)  e^{ - \mu_0 L_0}.    \]
        Then for sufficiently large $L_0$ the prefactor is bounded by 1, the second factor obtains the desired decay. \qed
   \end{proof}
  
  We will present an induction on the sequence of scales $L_0,L_1,L_2,...$, as defined in (\ref{Lscales}). The following condition is the base case on scale $L_0$ which follows from Proposition \ref{prop:bulk}.
  
    \begin{condition}[Initial regularity]\label{ic0}
    	Let $\kappa = C_\kappa(\mu_0) \epsilon_0 $ for any    $\mu_0 > 0$. If $B_{L_0,T_0} \cap \bW \subset \bL_\cV$ is  $\epsilon_0$-non-resonant then $x\in \bW $ is $ (\mu_0,L_0|\bW)$-regular.
    \end{condition}
    
    The following condition formulates the induction statement on the sequence of scales.
    
    \begin{condition}[Level $k$ regularity]\label{ick}
    	For  any $x\in \cV$ and $i = 0,1,..,k$, if $B_{L_i ,T_i}(x,t) \cap \bW $ is $\epsilon_{i}$ nonresonant then $(x,t)$ is $(\mu_i,L_i|\bW)$ regular.
    \end{condition}
    
    From Condition \ref{ick} we induct to regularity on the next level. 
    
    \begin{condition}[Level $k+1$ regularity]\label{ick1}
    	For  any $x\in \cV$  if $B_{L_{k+1} ,T_{k+1}}(x,t) \cap \bW $ is $\epsilon_{k+1}$-non-resonant then $(x,t)$ is $(\mu_{k+1},L_{k+1}|\bW)$ regular.
    \end{condition}
    
    It is helpful to prove the first induction step concretely before proceeding to higher scales.
    
    \begin{proposition}\label{P:ind0} 
    Suppose the environment process $\{ (\bbT^{\nu_i}, \bT_i, h_i ) \}_{i = 0 }^d$ is $(\zeta, \sigma) $-proper. Let $L_0$ be sufficiently large, and suppose Condition \ref{ic0} holds. Then,  if  $B_{L_{1},T_1}(x,t)\cap \bW$ is $\epsilon_{1} $ non-resonant,  $x$ is $(\mu_{1},L_{1}|\bW)$ regular. 
    \end{proposition}
    
    We will prove  Proposition \ref{P:ind0}  in Section \ref{sec:lev0}. Finally, we have the induction on higher scales which we will prove in   Section \ref{sec:levk}.

\begin{proposition}\label{P:indk}
    Suppose the environment process $\{ (\bbT^{\nu_i}, \bT_i, h_i ) \}_{i = 0 }^d$ is $(\zeta, \sigma) $-proper. Let $L_0$ be sufficiently large, and suppose   Condition \ref{ick} holds for given $k$. Then Condition \ref{ick1} holds.
\end{proposition}

   \subsection{Proper environments are regular}
   \label{proper reg}
     \begin{proposition} \label{res-sep}
    Suppose the environment process $\{ (\bbT^{\nu_i}, \bT_i, h_i ) \}_{i = 0 }^d$ is $(\zeta, \sigma) $-proper. Then there is a finite $L_0$ so that for any $x \in \cV $ and $L > L_0 $ the box $\Lambda_L^{\cV}(x)$ has at most $R_v$, respectively $ R_e$, many $ L^{-\alpha/\gamma}$-resonant sites, respectively edges.  
   \end{proposition}

   \begin{proof} 
     From the $\sigma$-admissible definition we have that there is $c > 0 $ so that for all $\wt\theta \in h_i^{-1}(0) $ we have
     \[   h_i(\theta') > c [d(\wt \theta,\theta')]^\sigma . \]
     Suppose for some $i$ there are points $x_j$, for $j = 1,..,R_i + 1$, so that $ h_i(\bM x_j + \theta) < L^{-\alpha}.$ Then there is some $\wt \theta \in h_i^{-1}(0) $ so that there are two values of $j$( say $j = 1,2$) so that 
     \[ 
          d(\bM_i x_j +  \theta, \wt \theta) 
         < \frac1c [ h_i(\bM_i x_j +  \theta)]^{\frac1\sigma}  
         < \frac1c L^{ - \frac{ \alpha }{ \sigma \gamma} } \] 
     On the other hand, from the Diophantine condition, 
     \[    \frac{ C_\zeta }{  \|x_1 - x_2\|^{\zeta} } 
         <
         d(\bM(x_1 - x_2),0) = 
        d(\bM_i x_1 + \theta, \bM_i x_2 + \theta) 
        \leq
        d(\bM_i x_1 +  \theta, \wt \theta) 
        + d(\bM_i x_2 +  \theta, \wt \theta) 
       . \]
     Thus, for some $C_1 > 0$,
     \[  
      C_1 L^{\frac{\alpha}{\sigma\gamma \zeta }} 
                                      < \|x_1 - x_2\|.  \]
     Thus, as $\alpha > \gamma\sigma\zeta,$ for large enough $L $, there are at most $R_0$ vertex  resonances and $R_1 + \cdots R_d $ edge resonances in any box $\Lambda_{L} ( x) $. \qed
   \end{proof}

   \begin{proposition}\label{pr:BC}
    Suppose the environment process $ \{ (\bbT^{\nu_i},\bT_i,h_i) \}_{i = 0}^d$ is $(\zeta, \sigma)$-proper. Then for any small enough $ \epsilon > 0 $. For almost any $(\theta_0,..,\theta_d)\in \bbT^{\nu_0 + \cdots + \nu_d }$ and any  $x$, there is $K_x $ so that for all $ k > K_x$ the region 
   $  \Lambda_{L_{k+1}^{1+\epsilon}}(x)$ is $(\mu_k,L_k)$-regular     
   \end{proposition}

   \begin{proof} Let $\nu = \min\{\nu_0,\nu_1,\ldots,\nu_d\}$,
    the probability a given site or edge is $\epsilon$-resonant is bounded by  $ C \epsilon^{ \nu / \sigma}$. Thus, the probability $\Lambda_{2 L_{k+1}^{1 + \epsilon}}(x) $ is $\epsilon_k$-resonant  is bounded by $C L_k^{ d(1+ \epsilon) \gamma  -  \nu \alpha/\sigma }$. But $ \Lambda_{2L_{k+1}^{1 + \epsilon}}$ being $\epsilon_k$ non-resonant implies  implies $ \Lambda_{L_{k+1}^{1 + \epsilon}}$  is $(\mu_k,L_k)$ regular by Proposition \ref{P:indk}.  Thus 
    \[  \sum_k P[\Lambda_{L_{k+1}^{1+\epsilon}}(x) \tn{ is not  } (\mu_k, L_k) \tn{ regular}]
    \leq \sum_k P[ \Lambda_{2L_{k+1}^{1 + \epsilon}} \tn{ is } \epsilon_k\tn{-resonant} ]  < \infty \]
    By the Borel Cantelli Theorem. there is $K_x$ so that $k> k_x$ implies $ \Lambda_{L_{k+1}^{1+\epsilon}}(x) $ is a $ (\mu_k,L_k)$regular region for all $ k > K_x$.\qed
   \end{proof}

  \begin{proof}[Proof of Theorem \ref{th:diophantine}]
    Let $b >> 2$ be a constant to be fixed below. Given $x$ let $K_x$ be the constant defined in Proposition \ref{pr:BC}. Let $C_x = \exp\{ \mu b L_{K_x + 1} \} $.  Suppose $k $ is such that
    \[  b L_{k} \leq | \max \{ |y - x| , |s-t|^{1/\eta} \} | < b L_{k+1}    \]
    if $k \leq K_x $ the Theorem is immediate, so suppose $k > K_x$. By Proposition \ref{pr:BC}, the region $\Lambda_{ b L_{k+1}}$ is $(\mu_k,L_k)$-regular.  Obtaining decay in the $(\mu_k,L_k)$ regular region is similar to the proof of Proposition \ref{prop:rhosmall}.  On the first step we have 
    \[  P[ (x,t) \lra (y,s) ] < 
    Q[ (x,t), B_{L_k, T_k}(x) ] \sup_{ (x',t') \in \partial^{H+} B_{L_k,T_k}(x)} P[(x',t') \lra (y,s)] \]
     As $\Lambda_{bL_{k+1}}(x) $ is $\epsilon_k$ non-resonant we have
     \[
     Q[(x,t),B_{L_k,T_k}(x)  ] < 
        e^{ - \mu_k L_k }.
     \]
   We apply the same argument $n$ times for
   \be \label{nmsa-BC}  n = \max \left\{\left\lfloor \frac{|y  - x|}{ L_k + 1 } \right\rfloor ,\left\lfloor \frac{| s - t|}{ L_k^\eta} \right\rfloor   \right\}  \ee
   which obtains 
   \[  P[(x,t) \lra (y,s)] \leq \left( \frac{\kappa}{\epsilon_k} e^{-\mu_kL_k} \right)^n \leq \exp \left( - \mu_k \left( 1 - \frac{ \alpha \log L_k}{ \mu_k L_k } \right) n L_k \right).  \]
   Now return to (\ref{nmsa-BC}). For $n$ defined by the time difference we have 
   \[  n = \left\lfloor \frac{ |s-t|  }{ L_k^\eta } \right\rfloor \geq \frac{|s- t| - L_k^\eta}{L_k^\eta} \geq \frac{ |t -s|^\tau }{ L_k}  \] 
   provided $b^{\eta} > 1 + b^{\tau\eta}$. For $n$ defined by spatial difference and $c = \frac{L_0}{ L_0+ 1} - b^{-1}$ we have 
   \[  n \geq  \frac{|y- x|  - (L_k +1)}{L_k+1} \geq c \frac{ |x-y| }{ L_k }  \]
   Thus we have the result 
   \[  P[(x,t) \lra  (y,s) ] < e^{-\mu\max\{|y-x|,|s-t|^{\tau} \}}  \]
   for $ \mu < \mu_\infty c \left( 1 - \alpha \frac{ \log L_0}{\mu_\infty L_0} \right)$. \qed   
  \end{proof}

   \section{Multiscale Analysis} \label{msa}
                  
  The strategy of multiscale analysis is to prove percolation regularity on an increasing sequence of scales and sets of phases. The initial lattice scale  $L_0$ must be chosen sufficiently large to allow the induction to proceed on higher scales. The tuning parameter $\kappa$ must be chosen sufficiently small to establish regularity on $\Lambda_{L_0}(0)$ for a large set of initial phases $\bth \in \bbT^{\bnu}$.

   We introduce parameters $p$ and $q$. Observe from (\ref{eta00}) that
   \[    \frac{(\gamma-1) \eta  - \gamma}{\alpha  \gamma }   >   2 R \frac{\gamma + R}{\gamma - R}  \] 
   and select $p$ such that
   \be\label{p00}
      \frac{(\gamma-1) \eta  - \gamma}{\alpha  \gamma } > p >   2 R \frac{\gamma + R}{\gamma - R}  
   \ee
   From the first inequality we have 
   \[  (\gamma - 1)\eta - \gamma > p\alpha\gamma  \]
   then select $q$ so that
   \be 
     \label{q00}
     \frac{(\gamma -1) \eta  - p\alpha \gamma }{\gamma} 
           > q > 1.
   \ee       

   \subsection{Percolation and low dimensional defects on the first level}\label{sec:lev0}
      
       In this section we prove Proposition \ref{P:ind0}. Thus let $x$ be chosen so that $\Lambda_{L_1}(x)$ is not $\epsilon_1$ resonant, we then show that $x$ is $(\mu_1,L_1) $ regular. By Proposition \ref{res-sep}, there are only $R_v$ sites and $R_e$ edges which are $\epsilon_0$ resonant in $\Lambda_{L_1}(x)$. Thus we will use that the bulk of the region is $(\mu_0, L_0)$ regular (for small enough $\kappa$) to control and separate the portion of the environment which is resonant.

       	We will break the proof of Proposition \ref{P:ind0}  into two parts: percolation to the edge  and vertex boundary. 
         First we state our bound on percolation to points in the inner edge boundary.

       \begin{proposition}\label{prop:Hbdry0}
       
       		Suppose  Condition \ref{ic0} holds and the environment $(\delta,\lambda)$ is $L_1$ nice. Then for   $(y,s) \in \partial^{\cE-} B_{L_1,T_1} $
            \[
    P(  (x,0)  \lra (y,s)| B_{L_1,T_1}(x))
      <   C L_0^{(\eta \gamma + \alpha) R }
       e^{ - \ol \mu_0 L_1}     
     \]    
     for $\ol \mu_0 = \mu_0 
     \left[1 - R  \frac{L_0}{L_1}\right]$ where $R  =  R_e + R_v $.
       \end{proposition}
       The proof of Proposition \ref{prop:Hbdry0} is contained in Section \ref{sec:lateral0}.
        
        \begin{proposition} \label{prop:Vbdry0}
        	Suppose $L_0$ is sufficiently large, the environment $(\delta,\lambda)$ is $L_1$ nice,  $\Lambda_{L_1}(x)$ is $\epsilon_1$ non-resonant, and Condition \ref{ic0} holds.
        	Then
        	\[ P[ \{\Upsilon_{L_1}^{\cV}(x,0)  \lra \Upsilon_{L_1}^{\cV}(x,\pm T_1) | B_{L_1,T_1}(x) \}] <  
        	 e^{- L_1^{q} }. \]   
        \end{proposition} 
        The notation $\Upsilon^\#_L(x,t)$ is defined in (\ref{eblock}). The proof of Proposition \ref{prop:Vbdry0} is contained in Sections \ref{sec:continuous1} and \ref{sec:layer1}.

        \begin{proof}[proof of Proposition \ref{P:ind0}]
         For this proof, let us write $ \bB = B_{L_1,T_1}(x,t)$.	First we consider $(y, \pm T_1) \in \partial^\cV B$,
         observe that
         \[ \{ (x,0) \lra (y,\pm T_1)|\bB \} \subset \{\Upsilon_{L_1}^{\cV}(x,0)  \lra \Upsilon_{L_1}^{\cV}(x,\pm T_1) |\bB \}  \]
          from Proposition \ref{prop:Vbdry0}, 
        	\begin{align} \label{eq:Vb0}
        	    \sum_{s = -T_1,T_1} \sum_{ y \in \Lambda_{L_1}(x)} P[(x,0) \lra (y,s) | B] \leq C_d L_1^d 
        	    e^{- L_1^{q} }.  
        	\end{align}
        	It is clear that
        	\[
        	\int_{ U \in \partial^\cE_{|W} B }
            \lambda_{U}  P[A \lra B \cap \Upsilon_{1/2} (U)| B]
            dU  
            \leq \frac{1}{\epsilon_1} C_d L_1^{d-1} T_1 \sup_{(y,s) \in \partial^{\cE-}}  
             P(  (x,0)  \lra (y,s)| B_{L_1,T_1}(x)).
        	\]
        	As $ q > 1$, for large enough $L_0$ this term dominates (\ref{eq:Vb0}), thus, using  Proposition \ref{prop:Hbdry0}
        	\[
        	  Q[(x,0), B ] 
        	  \leq
        	C_2 L_0^{(\alpha  + d - 1)\gamma + (\eta \gamma + \alpha) R }
       e^{ - \ol \mu_0 L_1}  
        	\] 
        	 Thus, for large enough $L_0$ we have the result.
        	  \qed
        \end{proof}

    \subsubsection{Percolation in the lattice directions}  \label{sec:lateral0}
          
       Throughout this section, we fix $x$ to be a site such that $\Lambda^\cV_{L_1}(x)$ is $\epsilon_1$-non-resonant.   We specify the $\epsilon_0$-resonant vertices and edges with the following notation,
        \[
         \cR_0^e = 
            \{ u \in \cE   \mid 
            \| u - x\| \leq L_1  ;
            \lambda_u  > \epsilon_0^{-1}  \}. 
        \]
        and
        \[ \cR_0^v = 
           \{  y \in \cV   
           \mid \|y - x\| \leq L_1;
            \delta_u < \epsilon_0  \}  
           \]
            From which we define the resonance set $ \cR_0 = \cR_0^e \cup \cR_0^v$.
            Let us denote the non resonant set as
            \[  \cW_1 
                = \Lambda_{L_1}(x) \setminus  \cV(\cR_g).    \]
           The non resonant region in percolation space is then
           \be \label{regular1}
             \bW_1 =   \cW_1 \times [-T_1,T_1]. 
           \ee
         
       \begin{proof}[proof of Proposition \ref{prop:Hbdry0}]  
          For this proof, again let $B = B_{L_1,T_1}(x) $.
          We will consider percolation paths that begin at $(x,0) $  and we filter the paths by their visits to segments of $\cR^g_0 \times [-T_1,T_1]$. 
          Let 
          \[\cA_r = \{(u_0,u_1,...,u_r):  u_0 = x  ;u_i \in \cR_0^e \tn{ or } u_i \in\cR_0^v \setminus \cV(  \cR_0^e)  \tn{ and } u_i \neq u_j \tn{ if } i\neq j   \}\] 
          be the collection of $r$-length non-repeating  sequences in $\cR_0$ with initial site $u_0 = x$. By the $L_1$ nice assumption, there are only $R_v$ sites and $R_e$ edges which are $\epsilon$-resonant so $\cA_r = \emptyset$ for $r >   R_e + R_v$.   
           For $\bA = (u_0,..,u_r)  \in \cA_r$ let us write 
           \[\bA_i = \cV (u_i) \times [- T_1,T_1]   \]
            where $\cV(\cdot)$ is defined in (\ref{verset}). For $(y,s)  \in \partial^{\cE-} B_{L_1,T_1}(x,0),$ we introduce the percolation event defined by the prescribed visits to the resonance lines, 
          \begin{align} \label{eq:VA} \Psi_\bA(y,s) :=\{ (x,0)\lra  \bA_1 | \bW_1 \} \circ \{  \bA_1 \lra \bA_2 | \bW_1 \} \circ \cdots \circ \{  \bA_r \lra (y,s)| \bW_1\} .\end{align}
    It is possible that $\bA_r$ is a boundary line of $B$, then if $(y,s) \in A_r$ the final event in (\ref{eq:VA}) holds trivially. Observe that the percolation to $(y,s)  \in \partial^{\cE-} B$ is contained in some percolation event $\Psi_\bA(y,s)$:
     \begin{align}
     \label{hpercUpaths}
     \{ (x,0)  \lra (y,s)| B\}   
      \subset 
      \cup_{r\geq 0}
      \cup_{\bA \in \cA_r} \Psi_\bA(y,s).     
     \end{align}
     Thus, to complete the proof, we need only control the events $\Psi_\bA$.     
          
     First let us observe that    edges incident to sets $A_i$ have crossing parmeter bounded by $\epsilon_0^{-1}$ so that,       
     \begin{align*}
         P(\{  \bA_i  \lra  \bA_{i+1}| \bW_1\} ) 
            &   \leq \kappa \epsilon_0^{-1}
              \int_{ U \in \partial^{\cE+} (\bA_i) } 
             P[ U \lra \bA_{i+1}  | \bW_1 ]  dU
          \\
          &  \leq \frac{4d\kappa T_1}{\epsilon_0}
             e^{- \mu_0 L_0
               \left \lfloor\frac{\dist(u_i,u_{i+1})  }{L_0} \right\rfloor} 
     \end{align*}
      Where we have used Proposition \ref{prop:rhosmall} on the second line.  Then  for any $\bA \subset \cA_0$,
          \[    \left \lfloor \frac{ \dist(u_1,u_0)}{L_0} \right\rfloor 
             + \left\lfloor  \frac{ \dist(u_2,u_1)}{L_0}\right \rfloor 
                + \cdots +  \left\lfloor  \frac{ \dist(u_r,u_{r-1})}{L_0}\right\rfloor 
                +  
               \left \lfloor \frac{ \dist(y,u_r) }{L_0}\right\rfloor 
                 >  \frac{L_1 - 4 r L_0}{L_0},
            \]
   Therefore, 
    we may uniformly bound any $\Psi_\bA$ event for $\bA \in \cA_r$,  
     \begin{align}
     \label{eq:lev0:H2}
       P( \Psi_\bA (y,s) ) \leq 
       \left (\frac{4d  T_1}{\epsilon_0} \right)^{ r }
       e^{ - \mu_0 \left(  L_1 - 4 r L_0  \right)}.
     \end{align}          
    Finally, observe there are at most $ { R \choose r} r! $ elements of $\cA_r$.
    Thus           
            \begin{align}
               \label{eq:lev0:H3}
    P( \{ (x,0) \lra (y,s) | B \} )
      <  e  R !  \left (\frac{4T_1}{\epsilon_0} \right)^{R}
       e^{ - \mu_0  \left[ 1 - 4  R  \frac{L_0}{L_1}  \right] L_1},     
     \end{align}
     as $T_1/\epsilon_0 = L_0^{\gamma \eta + \alpha}$,
     this is what we intended to prove. \qed
       \end{proof}

    \subsubsection{Percolation in the continuous direction}       \label{sec:continuous1}
       
       Percolating from $(x,0)$ to $\Upsilon_{L_1}(x,T_1)$ requires percolating through $\lfloor T_1/T_0 \rfloor $ intervals of length $T_0$. 
       Therefore we will prove the following:
        \begin{proposition}\label{prop:layer0}
        	Suppose $L_0$ is sufficiently large, the environment $(\delta,\lambda)$ is $L_1$ nice,  $\Lambda_{L_1}(x)$ is $\epsilon_1$ non-resonant, and Condition \ref{ic0} holds.  Then we have, for any $t $ so that $|t|< T_1 - T_0$,
        	\[  P[\Upsilon_{L_1}(x,t) \lra \Upsilon_{L_1}(x ,t+ T_0) | B_{L_1,\frac12 T_0}(x , t + \tfrac12 T_0) ]  
        	< 1 -  c^{R} \epsilon_1^{ 2 R }
        	\] 
        	were $c$ is a constant depending on $d$, $\mu_0$, and $C_\kappa$.
        \end{proposition} 
        The proof is contained in Section \ref{sec:layer1}.
        \begin{proof}[proof of Proposition \ref{prop:Vbdry0}]
        	To percolate from the layer at $t = 0$ to the layer at $t = T_1$ requires percolating through all layers $\Upsilon_{L_1}(x, k T_0)$ for $k = 1,2,.., T_1/T_0$.
        	Thus,  we have
        	\begin{align}\label{eq:vertp1}
        	  \{\Upsilon_{L_1}(x,0) \lra \Upsilon_{L_1}(x,T_1)|  B_{L_1,T_1}(x) \} 
        	 \subset
        	  \cap_{k=1}^{\lfloor T_1 /T_0 \rfloor} 
        	\{\Upsilon_{L_1}(x,(k-1) T_0) \lra \Upsilon_{L_1}(x,k T_0)   | B_{L_1,T_1}(x, (k  - \tfrac12)T_0 )\} . 
        	\end{align}
        	The events on the right hand side are mutually independent, therefore we apply Proposition \ref{prop:layer0} to each,
        	\begin{align} \label{eq:vertp2}
        	  P[ \{\Upsilon_{L_1}(x,0) \lra \Upsilon_{L_1}(x,T_1)  | B_{L_1,T_1}(x)\} ] 
        	 & \leq 
        	\left[ 1 - c^{R}\epsilon_1^{2 R}\right]^{T_1 / T_0 - 1}
           \leq
        	e^{- (T_1 / T_0 - 1)c^{R}\epsilon_1^{2R}}.
        	\end{align}
        	By the chosen parameters we have
        	\[  P[ \{\Upsilon_{L_1}(x,0) \lra \Upsilon_{L_1}(x,T_1)  | B_{L_1,T_1}(x)\} ]  
        	\leq
        	\exp\left\{- \frac12  c^{R+1}
        		L_1^{ \eta - \eta/\gamma -2 \alpha  R }  \right\},
        	\] 
        	for the chosen $q$ and large enough $L_0$ this completes the proof. 
        	\qed
        \end{proof}

          \subsubsection{Vertical percolation by layer}\label{sec:layer1} 
         
          We break the events of vertical percolation by layer (from $\Upsilon_{L_1}(x,kT_0)$ to $\Upsilon_{L_1}(x,(k+1)T_0)$) into two subevents. For simplicity we will carry out the proof in the $k = 0$ case. The first is percolation from $\Upsilon_{L_1}(x,0)$ to $\Upsilon_{L_1}(x,T_0)$ through the regular environment $\bW_1$, as defined in (\ref{regular1}). The second is percolation from the resonant lines $\cR_0\times [0,T_0]$ to $\Upsilon_{L_1}(x,T_0)$.     The resonant sets   are incident to $\Upsilon_{L_1}(x,T_0)$, thus, we control communication between these sets by requiring each resonant line has a death in the neighborhood of $\cR_0\times T_0$. 
          
         Now let us formalize these definitions. The thickness of the boundary layer will be different for the resonant edges and sites. Let 
         \be \label{xidef}  \xi  = \frac{\epsilon_1}{\kappa} \ee
           Let $I^{(line)} = [0,T_0 - \xi] $ and $I^{(bdry)} = [T_0 - \xi ,T_0].$ For $w \in \cR_0$, let 
         \[I_w^{(\#)} =   \cV(w) \times I^{(\#)},   \]
          for $\# = line,\ bdry$.
           For $ w \in \cR_0$, (either an edge or a site), we define an absence of  communication event from the line at $w$ as
        \[  F_{w}^{(1)} = \{ I_w^{(line)} \lra \Upsilon^\cV_{L_1}(x,T_0)\mid \bU_1 \}^c \]
        where 
        \[ \bU_1 = \bW_1 \cap  B_{L_1,\frac12 T_0}(x , t + \tfrac12 T_0).  \]
        For $w\in \cR_0$, define the  
         `break' event at the boundary as
        \[ 
          F_{w}^{(2)} = \left( \cap_{y \in \cV(w) } \{ (y\times I^{bdry} )  \cap \bbD \neq \emptyset \} \right)
           \cap
      \left( \cap_{u \in \cE(\cV( w))} \{( u \times I^{bdry} ) \cap \bbB  = \emptyset \} \right).  \] 
       The events of non-communication from the resonance line to the upper boundary, and break event near the upper boundary, are now
        \be
        \label{F12}
        F_\cR^{(1)} =   \cap_{u \in \cR_0 } F_{u}^{(1)}  \tn{ and }    F_{\cR}^{(2)} =   \cap_{u \in \cR_0 } F_{u}^{(2)} .
        \ee
      Now we may write the complement of percolation through the resonance set as,
        \[ F_{\cR} = F_\cR^{(1)} \cap F_\cR^{(2)}. \] 
      Absence of  percolation in the nonresonant region is written as   
           \[ G = \{ \Upsilon^\cV_{L_1}(x,0) \lra \Upsilon^\cV_{L_1}(x ,T_0) \mid \bU_1 \}^c. \]
          Note that $F_\cR $ and $G$ are both decreasing events and 
          $\{\Upsilon^\cV_{L_1}(x,0)  \lra  \Upsilon^\cV_{L_1}(x,T_0)|  \bU_1 \} \subset (F_\cR \cap G)^c $.

       \begin{proposition}\label{prop:res0a}
           Suppose  Condition \ref{ic0} holds and  $\Lambda_{L_1}(x)$ is $\epsilon_1$ non-resonant and Condition \ref{ic0} holds. Then  
           \[ P\left[ F_{ \cR}^{(1)}  \right] 
             \geq   c^R
              .  \]
             for $c$ depending on $d$, the choice of  $\mu_0$ and $C_\kappa$.
       \end{proposition}

       \begin{proof} Let $a > 0$, to be sufficiently small as determined later, and for all $ \ell \geq 0$ let us partition $I^{(line)}$ in to the intervals $ J^{(\ell)} =  T_0 - \xi - a \ell  +  [-a , 0 ]$. Then for any $u \in \cR_0$
     \[ 
          J_\cR^{(\ell)} = \cV( \cR) \times  J^{(\ell)} .
     \] 
     The probability of communication from such an interval   
     to $\Upsilon_{L_1}(x, T_0)$ is bounded by
       \begin{equation} \label{line2UPS}
       P(\{ J_\cR^{(\ell)} \lra \Upsilon_{L_1}^\cV (x,T_0) | \bU_1  \})
        < Q[J_\cR^{(\ell)};J_\cR^{(\ell)}| \bU_1]  
         \sup_{(y,s) \in \partial_{|\bW_1}^{\cE+} J_\cR^{(\ell)}} 
         P(\{(y,s) \lra \Upsilon_{L_1}^\cV (x,T_0)| \bU_1 \}). 
\end{equation}     
  By definition of the sets $J_u^{(\ell)}$, the incident  edges are not $\epsilon_0$ resonant. Thus 
       \[ Q[J_\cR^{(\ell)}; J_\cR^{(\ell)}| \bU_1 ] \leq \kappa \epsilon_0^{-1} 4d a |\cR_0| = C_\kappa4d a |\cR_0| \]
       where the right equality follows from definition of $\kappa$ (\ref{kappa}) 
       Therefore, using Lemma \ref{lem:liebfilter} and Proposition \ref{prop:bulk}, 
       \be
       \label{KAlra}
         P(\{ J_\cR^{(\ell)} \lra \Upsilon_{L_1}^\cV (x,L_1) | \bU_1  \})  \leq 
         C_\kappa 4d  a e^{-2\mu_0 (a\ell  + \xi )} |\cR_0|
       \ee  
       The communication probability from $I_u^{(line)}$ to $\Upsilon^\cV_{L_1}(x,  T_0)  $ is the union of the above events.
       Consider the complement of the communication events and use the FKG inequality,
       \[  P \left[  F_{u}^{(1)} \right] = P [  \cap_{\ell} \{J_\cR^{(\ell)}  \lra  \Upsilon_{L_1}^\cV(x,T_0) | \bU_1 \}^c ] \geq \prod_{\ell} P[ \{ J_\cR^{(\ell)} \lra  \Upsilon_{L_1}^\cV(x, T_0 ) | \bU_1 \}^c ]   \]
       Now apply (\ref{KAlra}) to obtain, 
        \[  P\left[  F_{\cR}^{(1)}  \right]\geq  
               \prod_{\ell \geq 0}  \left[ 1 - 
               C_\kappa 4da |\cR_0| e^{-2\mu_0(a \ell +\xi)}\right]    \]
         Thus, for small $a$, 
         \begin{align*}
        -\log    P\left[  F_{\cR}^{(1)}  \right]
        & \leq  
                  -   \sum_{\ell \geq 0} \log \left[ 1 - C_\kappa  4da |\cR_0|  e^{- 2 \mu_0 (a\ell +\xi) } \right] 
                  \\
        & \leq 
            \sum_{k \geq 0}    C_\kappa 8da \epsilon_0  |\cR_0| e^{- 2 \mu_0 (al+\xi)  } 
              =   C_\kappa  8d |\cR_0| e^{- 2\mu_0 \xi}  \frac{a}{1 - e^{- 2 \mu_0 a}} 
         \end{align*}
         therefore, taking $a$ sufficiently small, we have
         \[   P\left[  F_{u}^{(1)}  \right] \geq \exp\left\{ - C_\kappa\frac{ 8 d  }{\mu_0 }e^{  - 2\mu_0  \xi }  \right\} \geq   \exp\left\{ - C_\kappa\frac{ 8 d  }{\mu_0 } |\cR_0|  \right\}.\] 
         As $|\cR_0| \leq R$, this completes the proof. \qed  
         \end{proof}

       \begin{proposition}\label{prop:res0b}
           Suppose  Condition \ref{ic0} holds and  $\Lambda_{L_1}(x)$ is $\epsilon_1$ non-resonant and Condition \ref{ic0} holds. Then for $u \in \cR_0$,
         \[
         P\left[  F_{\cR_0}^{(2)} \right]  \geq 
          c^R \epsilon_1^{2 R}
         \]
       \end{proposition}
       
         \begin{proof}
         By definition, there are at most $ R_v$ sites and $R_e$ edges in $\cR_0$.
         Thus we need at most $R_v$ many $\epsilon_1$-non-resonant intervals and $ R_e $ many $\epsilon_1$-non resonant intervals at edges. Considering the neighboring  edges of the resonance set we have at most  $2d R_v + 4d R_e$ many $\epsilon_0$-non-resonant edges, the number of neighboring sites consist  of at most  $2R_e$ many $\epsilon_0$-non-resonant sites.
         Thus we have, for some $0 < c < 1$,
         \begin{align}
          P[F^{(2)}_{\cR}] \geq \prod_{w \in \cR_0} P[F_{w}^{(2)}] & \geq 
          \left( 1- e^{ - \epsilon_1 \xi /\kappa} \right)^{R_v}
           \left( 1- e^{ - \epsilon_0 \xi /\kappa} \right)^{2R_e}
           \exp\{ -(2d R_v + 4d R_e) \xi \kappa \epsilon_0 - R_e \xi \kappa \epsilon_1^{-1}  \} \\
           & \geq \nonumber
        c^R  \left( \frac12 \frac{\epsilon_1^2}{\kappa^2}\right)^{R_v}
           \left( \frac12 \frac{\epsilon_1 \epsilon_0}{\kappa^2}\right)^{ 2 R_e}
          \end{align}
         where the second line follows from the choice of $\xi$ (\ref{xidef}). The proof now follows immediately by the definition of $\kappa$ (\ref{kappa}).
         \qed
               
       \end{proof}

       \begin{proposition}\label{prop:res0c}
           Suppose  Condition \ref{ic0} holds and  $\Lambda_{L_1}(x)$ is $\epsilon_1$ non-resonant and Condition \ref{ic0} holds. Then for $A \in \cI_0$ 
           \[ P\left[ G  \right] 
             \geq  1 - C L_1^d e^{- \mu_0 L_0}  .   \]
       \end{proposition}
       \begin{proof}
       For each $y \in \Lambda_{L_1 }(x)$ apply Proposition \ref{eq:bulk}. 
       \[ P\left( (y,0) \lra  \Upsilon^\cV_{L_1}(x,T_0) \mid \bU_1 \right) 
         \leq e^{-\mu_0 L_0}.   \]
        The proof is completed by summing over all $y$ and taking the complement. \qed
       \end{proof}
       
       \begin{proof}[proof of  Proposition \ref{prop:layer0}]
          We will use the FKG inequality to bound the complement below
          \begin{align} \label{layerfkg}
               P[\{\Upsilon^\cV_{L_1}(x,0) \lra \Upsilon^\cV_{L_1}(x , T_0) | B_{L_1,\frac12 T_0}(x,\tfrac12 T_0)\}^c]  \geq  
               P[G ]   P[ F_\cR^{(1)} ] P[F_\cR^{(2)}].
              \end{align} 
       To bound the product, combine Propositions  \ref{prop:res0a}, \ref{prop:res0b} and \ref{prop:res0c}   which completes the proof.
         \qed
       \end{proof}

   \subsection{Percolation and low dimensional defects at higher levels.} \label{sec:levk}
   
   The arguments at the $k$ to $(k+1)^{th}$ level are essentially similar to the 0 to $1^{st}$ level argument. The main complication arises in the general versions of Propositions   \ref{prop:res0a} and \ref{prop:res0b}, the  percolation between horizontal layers. We will recapitulate the method in this section, but for similar proofs we only write down the necessary modifications.

      The following is the general version of Proposition \ref{prop:Hbdry0}.
      \begin{proposition}\label{prop:Hbdryk}
        Suppose the environment process $\{(\bbT^{\nu_i}, \bT_i, h_i)\}_{i = 0}^d $ is $(\zeta,\sigma)$-proper. Let $L_0 $ be sufficiently large and suppose  Condition \ref{ick} holds. Then for $(y,s)\in \partial_\bL^{\cE-} B_{L_{k+1},T_{k+1}}(x)$
      	\[
      	\bbP( \{ (x,0)  \lra (y,s)| B_{L_{k+1},T_{k+1}}(x)\} )
      	<   C L_{k+1}^{(\theta \gamma + \alpha)(R +1)}
      	e^{ - \ol \mu_k L_{k+1}}     
      	\]    
      	for $\ol \mu_k = \mu_k 
      	\left[1 - (R+1) \frac{L_k}{L_{k+1}}\right]$.
      \end{proposition}

    The following is the general version of Proposition \ref{prop:Vbdry0}.
 \begin{proposition} \label{prop:Vbdryk}
	  Suppose the environment process $\{(\bbT^{\nu_i}, \bT_i, h_i)\}_{i = 0}^d $ is $(\zeta,\sigma)$-proper. Let $L_0 $ be sufficiently large and suppose  Condition \ref{ick} holds.
	Then
	\[ P[ \{\Upsilon^\cV_{L_{k+1}}(x,0)  \lra \Upsilon^\cV_{L_{k+1}}(x,T_{k+1}) | B_{L_{k+1},T_{k+1}}(x) \}] <  
	e^{- L_{k+1}^{q} }. \]   
\end{proposition}

     The proof of Proposition \ref{P:indk} follows from Propositions \ref{prop:Hbdryk} and \ref{prop:Vbdryk} in a similar way to the proof of Proposition \ref{P:ind0}. 
     
     \begin{proof}[proof of Proposition \ref{P:indk}]
     Let us write $B = B_{L_{k+1},T_{k+1}}$. First let us consider the vertical boundary  $(y, \pm T_{k+1}) \partial^\cV B $,
     \[  \sum_{(y,s) \in\partial^\cV B  }
        P [(x,0) \lra (y,s)\mid B] \leq C_d L_{k+1}^d e^{- L_{k+1}^q} . \]
        Now let us consider the horizontal boundary:
        \[  \int_{U \in \partial^{\cE} B } \lambda_U 
             P[(x,0) \lra B \cap \Upsilon_{1/2}^\cV(U)\mid B]
              \leq C_d \frac{L_{k+1}^{d-1}T_{k+1} }{ \epsilon_{k+1}} \sup_{(y,s) \in \partial^{\cE-}} 
        P[ (x,0) \lra (y,s) \mid B] \]
        Again, as $q > 1$ and $L_0  > 1$ we have 
        \[  Q[(x,0), B] \leq C L_{k+1}^{(d-1) + \eta - \alpha}
            e^{- \ol\mu_k L_{k+1} }, \]
            which, for large enough $L_0$, completes the proof. 
     \end{proof}
     
     \subsubsection{Percolation through the lateral directions}     
     
  We define the sets  of $\epsilon_k$ edge and site resonances  
        \be
         \label{rsetek}
          \cR_k^e = 
         \left\{  u \in \cE\mid \| u - x\| \leq L_{k+1}; 
            \lambda_u > \epsilon_k^{-1}  \right\}. 
            \ee
     and
         \be
         \label{rsetvk}
          \cR_k^v = 
         \left\{  y \in \cV \mid  \| y - x \| \leq L_{k+1} ; 
            \delta_u < \epsilon_k   \right\}.    
            \ee
    From which we define the total resonance set $\cR_k = \cR_k^e \cup \cR_k^v$, the non-resonance set is 
    \[ \cW_{k+1} = \Lambda_{L_{k+1}}(x) \setminus \cR_k.   \]    
    The non-resonant region is 
    \[  \bW_{k+1} = \cW_{k+1} \times [-T_{k+1} , T_{k+1}] .\]    
    The proof of Proposition \ref{prop:Hbdryk} is similar to the proof of Proposition     
     \begin{proof}[proof of Proposition \ref{prop:Hbdryk}]
         As in Section \ref{sec:lateral0} let us define sets of sequences of length $r$,  $\cA_r$, which filter the visits of paths to the resonant set $B_{L_{k+1},T_{k+1}}(x) \setminus \bW_{k+1}$.  For $\bA\in \cA_r$ define $V_\bA$ in analogy to how it is defined in equation  (\ref{eq:VA}).
        As $\Lambda_{L_{k+1}} \setminus \cR_k$ is $\epsilon_k $ non-resonant, Condition \ref{ick} implies percolation in $\cB_k$ is $(\mu_k,L_k)$ regular. Thus, we may repeat the steps of Section \ref{sec:lateral0} to find, for any $\bA \in \cA_k$, 
    \[
       \bbP(V_\bA) \leq \left (\frac{4d R T_{k+1}}{\epsilon_k} \right)^{R_v + R_w + 1 }
       e^{ - \mu_k \left[  L_{k+1} - ( R +1 )L_k  \right]}.
     \]         
      Again, there are at most ${ R \choose r } r!$ elements of $\cA_r$ so the conclusion follows as in Section \ref{sec:lateral0}.
        \qed
        
     \end{proof}
     
     \subsubsection{Percolation in the continuous direction}

     Now we bound the percolation through a layer of the box. 
        \begin{proposition} \label{prop:layerk}
           Suppose  $\Lambda_{L_{k+1}}(x)$ is $\epsilon_{k+1}$ nonresonant and $\rho < \rho_0$ then, for sufficiently large $L_0$,
           \[  P[ \{  \Upsilon_{L_{k+1}}(x,0)  \lra   \Upsilon_{L_{k+1}}(x, T_k)  | B_{L_{k+1},T_{k+1}}(x)\}] <
                1 -  
                \epsilon_{k+1}^p
           \]
       \end{proposition} 
       We prove Proposition \ref{prop:layerk} in Section \ref{s:klayer}.
       \begin{proof}[proof of Proposition \ref{prop:Vbdryk}]
       	The proof is similar to
         the proof of Proposition \ref{prop:Hbdry0} in Section \ref{sec:continuous1}. 
          We again have,
         \[ \{ \Upsilon^\cV_{L_{k+1}} (x,0) \lra \Upsilon^\cV_{L_{k+1} }(x,T_{k+1}) \mid B_{L_{k+1},T_{k+1}}(x) \} \subset 
        \cap_{i = 0}^{\lfloor T_{k+1  } / T_k \rfloor }
          \{ \Upsilon^\cV_{L_{k+1}} (x,T_i) \lra \Upsilon^\cV_{L_{k+1} }(x,T_{i+1}) \mid B_{L_{k+1},T_{k+1}}(x) \}   \] 
          Thus we have 
          \[ P[  \{ \Upsilon^\cV_{L_{k+1}} (x,0) \lra \Upsilon_{L_{k+1} }(x,T_{k+1}) \mid B_{L_{k+1},T_{k+1}}(x) \}  ] 
           <  e^{- \epsilon_{k+1}^p \lfloor T_{k+1  } / T_k \rfloor   }  \]
           which, by the choice of $q$, the result follows for large enough $L_0$. \qed 
       \end{proof}

        \subsubsection{Percolation per layer}\label{s:klayer}
        In this section we will redefine some notation, so that $\cR_i$ are redefined on each level $k$ of the proof.

        Starting with $k$ level  resonances defined in (\ref{rsetek}) and (\ref{rsetvk}) we define neighborhood resonant lines inductively, for $i = k-1,..,1,0$. For all $i$, let $ \cR_{i}'' $ be the `complete set' of $\epsilon_{i}$ resonances, with
        \[\cR_i^{e \prime\prime }     
         = \{ u \in \cE \mid \|u - x\|  \leq  L_{k+1}  ;    \lambda_u  >  \epsilon_{i}^{-1}  \} \]
        and
        \[
        \cR_i^{v\prime\prime }= \{ y \in \cV \mid \| y - x\|  \leq  L_{k+1}  ; \delta_y < \epsilon_{i}   \} 
        \]
        and $\cR_i'' = \cR_i^{e \prime\prime }  \cup  \cR_i^{v\prime\prime } $.
        Then,  let $\cR_{k-1}'$ be the smallest set so that $\cR_k \subset  \cR_{k-1}^{\pe} \subset \cR_{k-1}^{\pe\pe}$, and $ \cR_{k-1}^\pe $ is $2L_{k-1}^\zeta $ nonintersecting the remainder of  $\cR_{k-1}''$: 
        \[ \dist( \cR_{k-1}' , \cR_{k-1}'' \setminus  \cR_{k-1}' ) >  2 L_{k-1}^\zeta. \]
       Finally let $\cR_{k-1}  = \cR_{k-1}^\pe\setminus \cR_{k} $. We  continue a similar induction to lower scales, so that $\cR_{i}^\pe $ is the smallest set so that  $\cR_{i+1}^\pe 
        \subset  \cR_{i}^{\pe} \subset \cR_{i}^{\pe\pe}$, and $ \cR_{k-1}^\pe $ is $L_{i}^\zeta $ nonintersecting the remainder of  $\cR_{i}^{\pe\pe}$: 
        \[ \dist( \cR_{i}' , \cR_{i}'' \setminus  \cR_{i}' ) >  2 L_{k-1}^\zeta. \]
        From the construction, $\cR_0' = \cup_{i=0}^{k} \cR_i$.
     
     We will repeat the structure of the proof from Section \ref{sec:lev0}. Let us re-partition $\cR_0'$ in the following way. 
     Let $Z_k$ be the subset of $\cR_0'$ which is 2 percolating from $\cR_k$.  Inductively, let $Z_i$ be the subset of $\cR_0'\setminus \left( \cup_{j = i+1}^{k} Z_j  \right)$ which is 2 percolating from $\cR_i$.

 Let us define the resonance lines, for $w \in Z_i$, as
 \[  I_w^{(line)} = \cV(w) \times [0,T_k - \xi_i ] \tn{ and }
     I_w^{(bdry)} = \cV(w) \times [T_k - \xi_i , T_k] .\]
        Now  we denote a reduced environment as       
       \be \label{nonresk}
        \bU_{k+1} = 
         \left(
          \Lambda_{L_{k+1}}( x) \setminus  \cV( \cR_0^\pe )  \right)   \times[0,T_{k}]  
       \ee  
     Note that $\cR_0'$ is $L_{k-1}^\zeta $ non-intersecting $\epsilon_{k-1}$ resonant sites and bonds in $\Lambda_{L_{k+1}}(x) \setminus \cR_0'$. We split percolation from $\Upsilon^\cV_{L_{k+1}} (x,0)$ to  $\Upsilon^\cV_{L_{k+1}} (x,T_k)$  through $\bU_{k+1}$ without visiting the resonant set and percolation from $\cup_{ w \in \cR_0'} I_w^{(line)}$ to $\Upsilon^\cV_{L_{k+1}} (x,T_k)$ through $\bU_{k+1}$.
     
      Define an absence of  communication event, 
        \be\label{FAki}  F_{u}^{(1)} = \{ I_u^{(line)} \lra \Upsilon^\cV_{L_{k+1}}(x,T_k)\mid \bU_{k+1} \}^c,\ee
         and a `break' event, for $w \in \cR_i$, 
        \[ 
          F_{w}^{(2)} = \left( \cap_{y \in \cV(w)} \{ (y \times I^{(bdry)} )  \cap \bbD \neq \emptyset \} \right)
           \cap
           \left( \cap_{u \in \cE(\cV(w))}
      \{ (u \times I^{(bdry)})  \cap \bbB = \emptyset  \}\right),  \]
      specifying at least one death and no bonds incident to $I_w^{bdry} $. Similar to Section \ref{sec:lev0} we define the family of events 
        \be
        \label{FAk}
        F_{\cR}^{(1)} = \cap_{w \in \cR_0^\pe}   F_{w}^{(1)} 
        \tn{ and }
       F_{\cR}^{(2)} = \cap_{w \in \cR_0^\pe } F_w^{(2)}.  \ee
     We reuse the notation for absence of  percolation in the nonresonant region, so that   
           \[ G = \{ \Upsilon_{L_{k+1}}(x,0) \lra \Upsilon_{L_{k+1}}(x , T_k)| \bU_{k+1} \}^c. \]
           We have again prepared decreasing events $F_\cR^{(i)}$ and $G$ such that percolation through the layer is contained in the complement of the intersection, so that we may bound percolation probability using the FKG inequality.

          First let us make an observation on the regularity of sites in proximity to $\cR_0'$.

          \begin{proposition}
          	\label{prop:reduced regular}
          	  For $0 < s< T_0$, $y \in \partial^{+} \cR_i$
             \be \label{e0}
                 P((y,T_k - s) \lra \Upsilon^\cV_{L_{k+1}}(x,T_k) | \bU_{k+1}  )  < e^{ - 2\mu_0 \min \{ s, L_0^\gamma\}}
             \ee
             for higher scales, $ i = 0,...,k-1$, and $ T_i \leq  s \leq T_{i+1}$
             \be \label{ei}
                 P((y,T_k - s) \lra \Upsilon^\cV_{L_{k+1}}(x, T_k ) | \bU_{k+1} )  < e^{ - \mu_i L_i}
             \ee
          	
          \end{proposition}
           Note the events described in (\ref{FAki}) cover percolation within this region.
          \begin{proof}
             By construction, for any $i$ and $u \in \cR_i$ we have, for $j > i$, $\dist(u,\cR_j) \leq 2(L_i^\zeta +\cdots L_{j-1}^\zeta )$   so that $ u $ is $L_j^\zeta$ separated from $\cR_j''\setminus \cR_j'$. For $ j < i$, $u $ is $2 L_j^\zeta$ separated from $\cR_j'' \setminus \cR_j'$ by construction. Thus we have that $\bU_{k+1} $ is $\epsilon_0$ non-resonant from Proposition \ref{prop:bulk}, we have (\ref{e0}). For $0 <  i  < k-1$ we have 
              $\bU_{k+1}$ is $\epsilon_i$ non-resonant so that, for $ T_i \leq  s \leq T_{i+1}$, we have that $(y,T_k - s)$ is $(\mu_i,L_i| \bU_{k+1} )$ regular.\qed
           \end{proof}

       \begin{proposition}\label{prop:resk}
           Suppose  Condition \ref{ic0} and \ref{ick} hold. Then  
           \[ P\left[ F_{\cR}^{(1)}  \right] 
             \geq  c^{  |\cR_0^\pe| }  \]
             for $c$ depending only on the choice of $\mu_0$, $d$ and $C_\kappa$. 
       \end{proposition}

       \begin{proof}
        For  $a > 0 $ and $\ell = 0,..,(T_k - \xi_i)/a $ define the intervals  
     \[  J_{\cR}^{(\ell)} = \cV( \cR_0^\pe ) \times  J^{(\ell)} \tn{ where }
     J^{(\ell)}  =   T_k - r_i - a\ell  + [ - a , 0  ]  .
     \]
     The proof is similar for all $i$, so we will not make any distinctions in the proof below. 
     
     For  $0 \leq \ell \leq (T_0-\xi_i )/a$, we apply (\ref{e0}), so that
     \begin{align}
     \label{T0}
     P( J^{(\ell)}_\cR \lra \Upsilon_{L_{k+1}}(x,T_k)|\bU_{k+1} ) 
     <   C_\kappa  \kappa 4 d |\cR_0^\pe|  a 
      e^{ -2 \mu_0  \min\{ \xi_i + al, L_0^\zeta\}} 
     \end{align}
      For $j = 0,1,..,k-1$ and $\ell$ so that $T_j <  a\ell + \xi_i < T_{j+1}$ we apply (\ref{ei}), so that
     \begin{align}
     \label{Ti}
     P( J^{(\ell)}_Z \lra \Upsilon_{L_{k+1}}(x,T_k)|\bU_{k+1} ) 
     < C_\kappa  4d |\cR_0^\pe| a 
      e^{ - \mu_i L_i } 
     \end{align}
   We now   combine  these bounds and apply the FKG inequality to find   
       \[  P \left[  F_{Z}^{(1)} \right] 
       = P [  \cap_{ \ell } \{   J^{(\ell)}_\cR  \lra  \Upsilon_{L_1}(x,T_k) 
       | \bU_{k+1} \}^c ] \geq 
       \prod_{\ell} 
       P[ \{  J^{(\ell)}_\cR \lra  \Upsilon_{L_1}(x,T_k)
        | \bU_{k+1} \}^c ]   \]
       Using the bounds (\ref{T0}) and (\ref{Ti}) we have
       \begin{align} \nonumber 
        - \log  P \left[  F_{\cR}^{(1)} \right] 
            &   \leq
            \sum_{\ell = 0}^{ (T_0 - \xi_i)/a } -\log \left(1 -   C_\kappa 4d |\cR_0^\pe|  a e^{-2\mu_0 \min\{ r_i + al, L_0^\zeta\}} \right) 
            +
            \sum_{i = 0}^{k-1} - \frac{T_{i+1 }} |\cR_0^\pe |a \log \left( 1 -  C_\kappa  2d  a e^{-\mu_i L_i } \right)
                      \end{align}
         We first bound the  second summation, for small enough $a > 0$ and sufficiently large $L_0$,
                 \[
            \sum_{i = 0}^{k-1} - \frac{T_{i +1 }}a \log \left( 1 -    C_\kappa 4d |\cR_0^\pe|  a e^{-\mu_i L_i } \right)
            \leq  8d |\cR_0^\pe| C_\kappa T_{1} 
            e^{-\mu_0 L_0}
                 \]     
      The first term is handled similarly to the first level proof,
      \[   \sum_{\ell = 0}^{ (T_0 - \xi_i)/a }  -\log \left(1 -  C_\kappa 4d |\cR_0^\pe|  a e^{-2\mu_0 \min\{ \xi_i + al, L_0^\zeta\}} \right)  <  8 d C_\kappa  |\cR_0^\pe| e^{-2 \mu_0 \xi_i}  \]
         Thus we have
         \[   P\left[  F_{A}^{(1)}  \right] \geq \exp\left\{ -  C_\kappa \frac{8d  }{\mu_0 }|\cR_0^\pe|e^{  - 2\mu_0 \xi_i }  \right\}  \geq   \exp\left\{ -  C_\kappa \frac{8d  }{\mu_0 }|\cR_0^\pe|  \right\} . \] 
         \end{proof}

       \begin{proposition}\label{prop:resib} There is some $c > 0$ so that
         \[
         P\left[  F_{\cR}^{(2)} \right]  \geq 
         c^{|\cR_0^\pe|}  \epsilon_{k+1}^{2 |\cR_k|} 
         \prod_{i = 0}^{k-1} 
         \epsilon_{i+1}^{2( |\cR_{i+1} | R + |\cR_i| )}
         \]
       \end{proposition}
       
         \begin{proof} 
         We begin the proof with the resonances in $Z_k$. For each $w \in Z_k \cap \cR_k$ there are at most $ R_e$ edges and $R_v$ sites of $ \cR_0' \setminus \cR_k  $ 2 percolating from $w$. We can assume elements of $ \cR_0' \setminus \cR_k  $  are not $\epsilon_{k }$ resonant. Now using bounds on $F_w^{(2)}$ as in (...), we have 
         \[   P(\cap_{w \in Z_k} F_{w}^{(2)} ) \geq
          \left(   \prod_{w \in Z_k  } P( F_{w}^{(2)} )  \right)  \]
         For $y \in \cV\cap Z_k$ we have,
         \[ P( F_y^{(2)}) = (1 - e^{- \xi_k \kappa^{-1} \delta_y })
                        \left( \prod_{u\in \cE(w)}  e^{ -\xi_k \kappa \lambda_u  }  \right)  \]
         For $u \in Z_k$ we have, 
         \[  P (F_u^{(2)}) =  e^{ - \xi_i \kappa \lambda_u }
            \left(\prod_{y\in \cV(u)}(1 - e^{- \xi_i \kappa^{-1} \delta_y })\right)
           \left( \prod_{w \in \partial^{\cE+}\cV(u) } e^{ - \xi_i \kappa \lambda_w } \right)   \]
         There are at most $R_v$ sites in $\cV \cap Z_k$ and   $R_e$ bonds in $\cE \cap Z_k$. Each such vertex and edge 2 percolates to at most $R_v$ vertices and $ R_e $ edges which are not $\epsilon_k $ resonant.
        Combining these, for some $c > 0$ independent of $k$, 
    \[    P(\cap_{w \in Z_k} F_{w}^{(2)} ) \geq
      c^{|\cR_k|} \epsilon_{k+1}^{2 |\cR_k| } \epsilon_{k}^{2R|\cR_k|}   \]    
       as $\kappa $ has been chosen sufficiently small.
      Continuing this argument for $ i = k-1,...,1$ we have 
           \[    P(\cap_{w \in Z_i} F_{w}^{(2)} ) \geq
      c^{|\cR_i|} \epsilon_{i+1}^{2 |\cR_i| } \epsilon_{i}^{2R|\cR_i|}   \]    
       Finally, the $i = 0$ case is similar to the initial case so that 
       \[    P(\cap_{w \in Z_0} F_{w}^{(2)} ) \geq
      c^{|\cR_0|} \epsilon_{1}^{2 |\cR_0| }   \] 
        From the FKG inequality, we have   the observation that
         \[ P(F_\cR^{(2)}) \geq \prod_{w \in \cR'_0}  P (F_w^{(2)} )  =  \prod_{i = k}^{0}   \prod_{w \in \cR_i}  P (F_w^{(2)} ).  \]
         The proposition now follows from FKG inequality and combining all above bounds on products.
         \qed
               
       \end{proof}

       \begin{proposition}\label{prop:reskc}
           Suppose  Condition \ref{ic0} holds and  $\Lambda_{L_{k+1}}(x)$ is $\epsilon_{k+1}$ non-resonant and Condition \ref{ic0} holds. Then  
           \[ P\left[ G  \right] 
             \geq  1 - C L_{k+1}^d e^{- \mu_k L_k}  .   \]
       \end{proposition}
       \begin{proof}
       For each $y \in \Lambda_{L_1 }(x)$ apply Proposition \ref{eq:bulk}. 
       \[ P\left( (y,0) \lra  \Upsilon_{L_{k+1}}(x,T_k) \mid \cW^\circ_{k+1} \right) 
         \leq e^{-\mu_k L_k}.   \]
        The proof is completed by summing over all $y$ and taking the complement. \qed
       \end{proof} 

    \begin{proof}[proof of Proposition \ref{prop:layerk}]
    
          By Proposition \ref{res-sep}, there are only $R$ points  in $\cR_k$. For each $u \in \cR_k$ there are at most $R$ points in $\Lambda_{L_k}(u)$ which are $\epsilon_{k-1}$ resonant, as $2 R L_{k-1}^{\zeta} < \frac{1}{2} L_k$ the $\epsilon_{k-1} $ resonant points can not percolate outside the set  $\Lambda_{L_k}(y)$. Thus $\cR_{k-1}$ contains at most $R_g^2$ points, note we are not including the points of $\cR_{k}$ in $\cR_{k-1}$. A similar argument implies, for each $w\in R_{j}$, that there are at most $ R $ elements in $R_{i}$ for $i < j$.
          Thus there are at most $R( R + R^2)$ many  $\epsilon_{k-2}$ resonant points which $L_{k-2} $ percolate from $\cR_{k} \cup \cR_{k-1}$. Formally $|\cR_{k-2}| \leq R^2 +  R^3$. 
          Continuing to further scales obtains there are at most $ R^2  + \cdots + R^{j+1} = R^2 \frac{R^j -1  }{ R -1} $ many $\epsilon_{k- j}$ resonant sites percolating from $\cR_{k-j+1}' =\cR_k \cup \cR_{k-1} \cup \cdots \cup \cR_{k-j+1} $.  Thus, for each $i$ we have  
          \be \label{cricount} |\cR_{i}| \leq   R^2 \frac{R^{k-i}  -1 }{ R - 1}  \leq  2  R^{k +1 -i}   \ee 
          (recall  $  R \geq 2 $). It follows immediately that $ |\cR_0^\pe | \leq 4 R^{k+1}$.
          
           Using (\ref{cricount}) and the bound  in Proposition \ref{prop:resib} we have
         \[
         P\left(  F_{\cR}^{(2)} \right)  \geq  
         c^{4 R^{k+1} }
          \epsilon_{k+1}^{2 R } 
         \prod_{i = 0}^{k-1} 
         \epsilon_{i+1}^{4 R^{k+1 - i}} 
         \]
         By definition $\epsilon_{i+1} = L_0^{- \alpha \gamma^{i+1}}$ thus,  
            \[
         P\left(  F_{\cR}^{(2)} \right)  \geq  
         c^{4 R^{k+1} }
          \epsilon_{k+1}^{2 R  + 4 \frac{ R^2}{\gamma - R} }  
         \]
         Now using Proposition  \ref{prop:resk} and the above bound, there is $c_1 > 0$ so that 
         \[  P\left( G \cap F_\cR^{(1)} \cap F_{\cR}^{(2)} \right) 
             > \frac12 c_1^{ 4 R^{k+1}} 
        \epsilon_{k+1}^{2 R  + 4 \frac{ R^2}{\gamma - R} }   \]      
        As 
        \[  \{  \Upsilon^\cV_{L_{k+1}}(x,0)  \lra   \Upsilon^\cV_{L_{k+1}}(x, T_k)  | B_{L_{k+1},T_{k+1}}(x)\}^c \supset   G \cap F_\cR^{(1)} \cap F_{\cR}^{(2)}  \]   
       by the choice of $p$, and for large enough $L_0$, this completes the proof.
    \qed
    \end{proof}

\bibliographystyle{plain}	
\bibliography{noterefs}

\end{document}